\documentclass[preprint,12pt]{elsarticle}
\usepackage[OT1]{fontenc}
\usepackage{pslatex}
\usepackage{a4wide}
\usepackage{color}
\usepackage{graphicx}
\usepackage{color}
\usepackage[latin1]{inputenc}
\usepackage{psfrag}
\usepackage{pdfsync}

\usepackage{amsfonts}
\usepackage{amsmath}
\usepackage{amssymb}
\usepackage{amsthm}
\usepackage{enumerate}
\numberwithin{equation}{section}
\newtheorem{Theorem}{Theorem}[section]
\newtheorem{Lemma}[Theorem]{Lemma}
\newtheorem{cor}[Theorem]{Corollary}
\newtheorem{Proposition}[Theorem]{Proposition}

\theoremstyle{remark}

\def\dP{\tilde{P}}
\def\dE{\tilde{E}}
\def\dX{\tilde{X}}

\newcommand{\R}{{\mathbb R}}

\newcommand{\Z}{{\mathbb Z}}
\newcommand{\N}{{\mathbb N}}
\newcommand{\Pe}{{\mathcal P}}
\newcommand{\PP}{{\mathbb P}}
\newcommand{\EE}{{\mathbb E}}

\newcommand{\E}{{\mathcal E}}
\def\M{\mathcal M}
\def\T{\mathcal T} 
\newcommand{\Se}{S}

\def\tT{{T}}
\def\pP{{P}}
\def\eE{{E}}

\def\Q{{\mathcal Q}}

\renewcommand{\S}{{S^{\circ}}}

\def\B{\mathcal B}
\def\rw{B^0_{[t,0]}}

\newcommand{\Div}{{\rm div}}
\newcommand{\sign}{{\rm sg}}

\newcommand{\e}{{\rm e}}
\newcommand{\Sm}{{S^{\rm o}}}

\newcommand{\one}{{\mathbf 1}}

\renewcommand{\paragraph}[1]{\noindent {\bf #1}$\quad$}

\def\R{\mathbb{R}}
\def\C{\mathcal{C}}

\def\H{\mathcal H}

\def\cI{{\mathcal I}}
\def\cJ{{\mathcal J}}
\def\bb{{b}}
\def\uu{{u}}
\def\vv{{s}}
\def\ww{{s'}}
\def\ss{{\mathbf s}}
\def\sso{{\mathbf s^{\circ}}}

\def\Vor{{\rm C}}
\def\Cen{{\rm Cen}}

\def\Nlf{\mathcal{N}}
\def\Vor{{\rm Vor}}
\def\Cen{{\rm Cen}}

\begin{document}

\begin{frontmatter}

\title{Harmonic deformation of Delaunay triangulations}

\author[uba]{Pablo A. Ferrari\corref{cor1}}
\ead{pferrari@dm.uba.ar}
\ead[url] {http://mate.dm.uba.ar/$\sim$pferrari}

\author[uba]{Pablo Groisman\corref{cor1}}
\ead{pgroisma@dm.uba.ar}
\ead[url]{http://mate.dm.uba.ar/$\sim$pgroisma}

\author[usp]{Rafael M. Grisi\corref{cor2}\fnref{fn1,fn2}}
\ead{rafael.grisi@ufabc.edu.br}

\address[uba]{\indent Departamento de Matem\'atica\break
\indent Facultad de Ciencias Exactas y Naturales\break
\indent Universidad de Buenos Aires \break
\indent Pabell\'{o}n I, Ciudad Universitaria\break
\indent C1428EGA Buenos Aires, Argentina.}

\address[usp]{Instituto de Matemática e Estatística\break
\indent Universidade de São Paulo\break
\indent Rua do Matão, 1010. Cidade Universitária\break
\indent S\~ao Paulo, SP, Brasil, CEP 05508-090.}

\cortext[cor1]{Corresponding author}
\cortext[cor2]{Principal corresponding author}

\fntext[fn1]{Present Adress:
Centro de Matem\'atica, Computa\c c\~ao e Cogni\c c\~ao, Universidade Federal do ABC,\\ Av do Estado, 5001, Santo Andr\'e, SP, Brasil, 09210-910}

\fntext[fn2]{Phone: (+55 11) 4996 8123.}

\begin{abstract}
  We construct harmonic functions on random graphs given by Delaunay
  triangulations of ergodic point processes as the limit of the
  zero-temperature harness process.
\end{abstract}


\end{frontmatter}

\section{Introduction}

Let $S$ be an ergodic point process on $\R^d$ with intensity $1$ and $\S$ its
Palm version. Call $\Pe$ and $\E$ the probability and expectation associated to
$S$ and $\S$ (we think that $S$ and $\S$ are defined on a common probability
space). The Voronoi cell of a point $s$ in $\S$ is the set of sites in $\R^d$
that are closer to $s$ than to any other point in $\S$. Two points are
\emph{neighbors} if the intersection of the closure of the respective Voronoi
cells has dimension $d-1$. The graph with vertices $\S$ and edges given by pairs
of neighbors is called the Delaunay triangulation of $\S$. The goal is to
construct a function $H:\S\to\R^d$ such that the graph with vertices $H(\S)$ and
edges $\{(H(s),H(s')),\,$ $s$ and $s'$ are neighbors$\}$ has the following
properties: (a) each vertex $H(s)$ is in the barycenter of its neighbors and (b)
$|H(s)-s|/|s|$ vanishes as $|s|$ grows to infinity along any straight line. If
such an $H$ exists, the resulting graph is the \emph{harmonic deformation} of
the Delaunay triangulation of $\S$. The search of such $H$ has been proposed by
Biskup and Berger \cite{Biskup}, who proved its existence in the graph induced
by the supercritical percolation cluster in $\Z^d$; their approach was the
motivation of this paper. The harmonic function $H$
was tacitly present in Sidoravicius and Sznitman \cite{ss} and in Matthieu and
Piatnitski \cite{Mathieu}; the function $H(s)-s$ is called \emph{corrector}. See
also Caputo, Faggionato and Prescott \cite{CFP} for a percolation-type graph in
point processes on $\R^d$. 

\begin{figure}[th]
\begin{center}
\[
\begin{array}{cc}
\includegraphics[height=6cm]{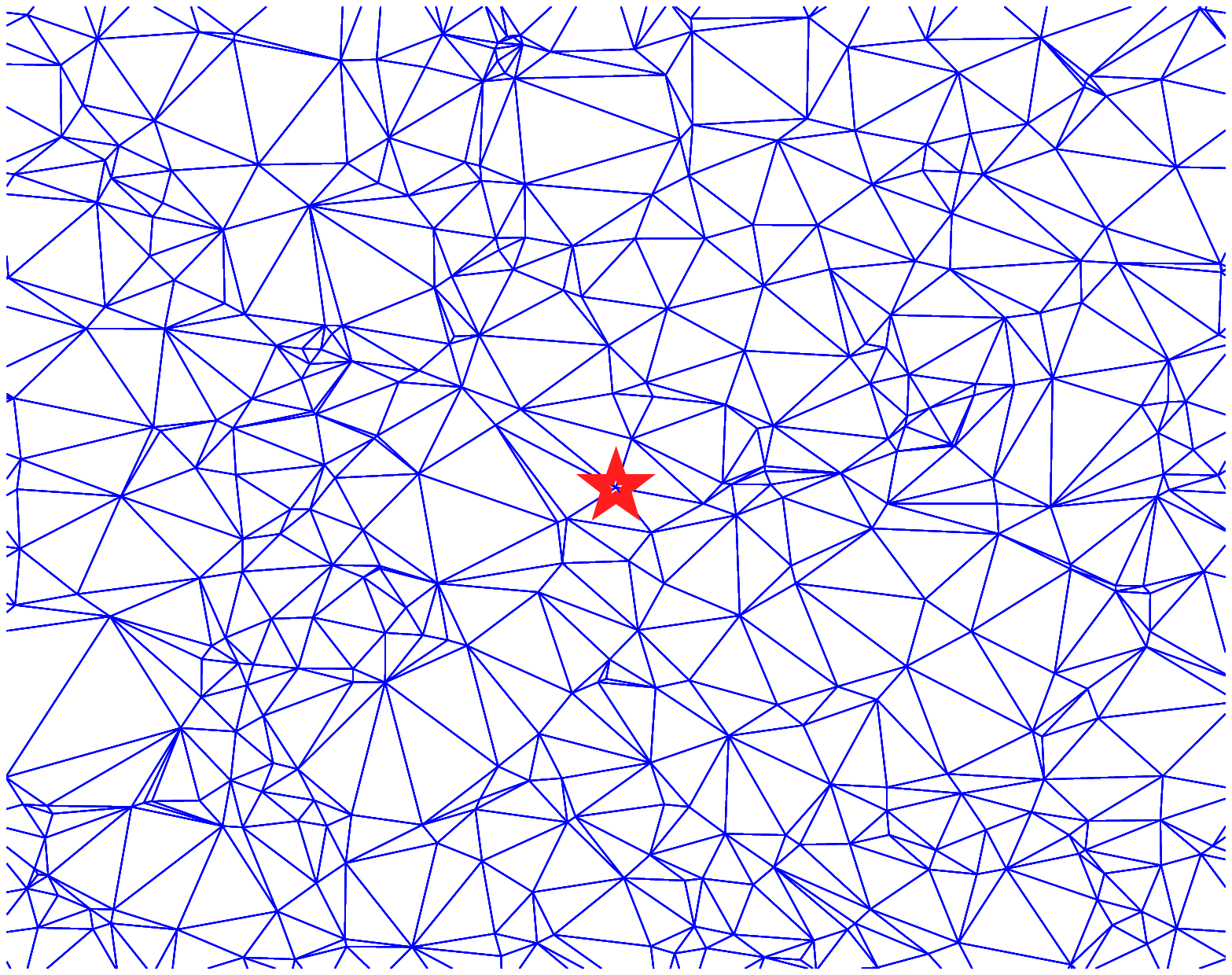} &
\includegraphics[height=6cm]{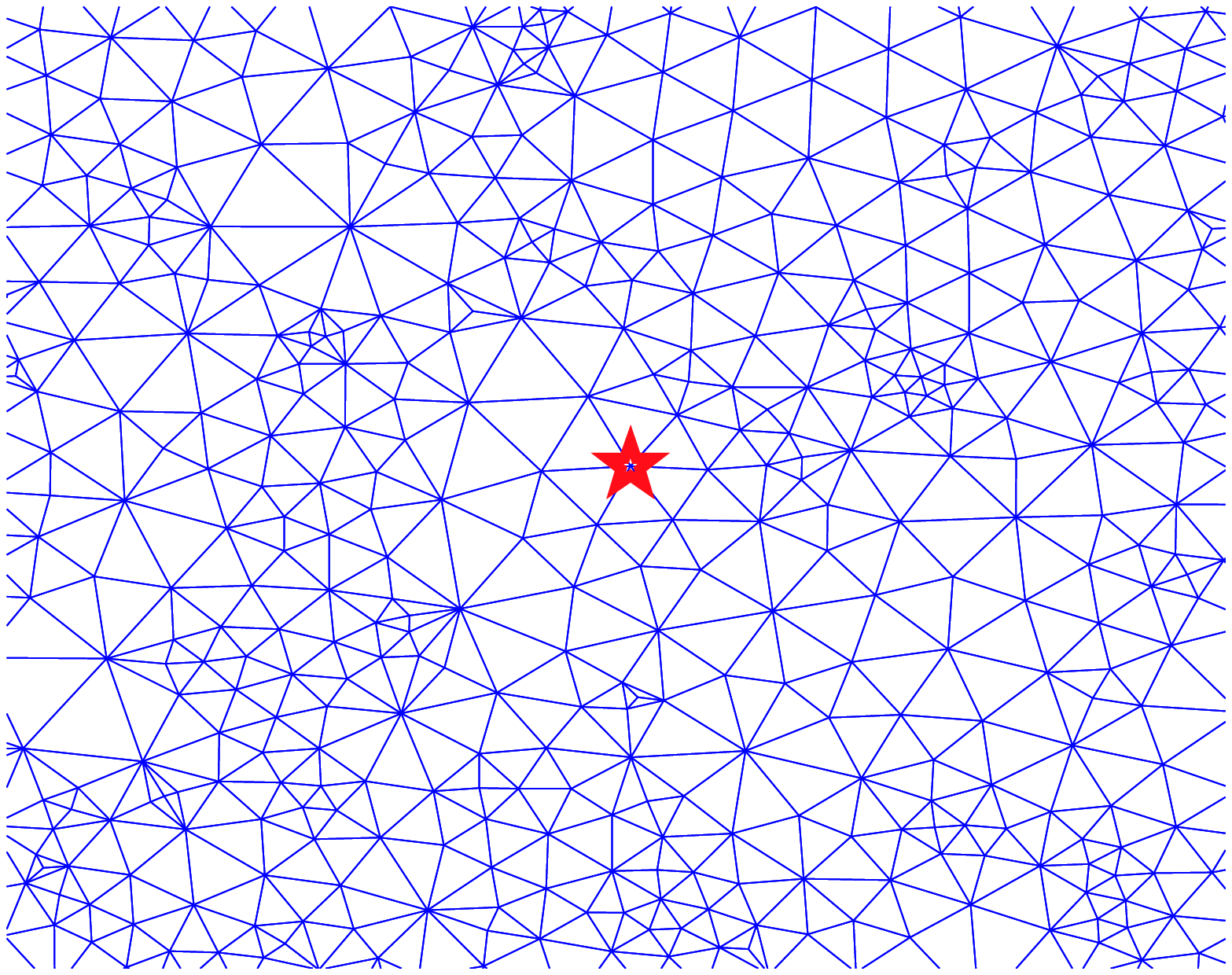}
\end{array}
\]
\caption{Delaunay triangulation of a Poisson process and its harmonic deformation. The star indicates the origin (left) and the point $H(0)$ (right).}
  \label{fig:2s}
\end{center}
\end{figure} 

The functions from $\S$ to $\R$ are called \emph{surfaces}. The coordinates
$h_1,\dots, h_d$ of $H$ are \emph{harmonic} surfaces; that is $h_i(s)$ is the
average of $\{h_i(s')$, $s'$ neighbor of $s\}$. The sublinearity of the corrector,
requirement (b) above, amounts to ask that $h_i$ have \emph{tilt} $e_i$, the
i-th canonical vector of $\R^d$.  Roughly speaking, a surface $f$ has
\emph{tilt} $u$ (a unit vector) if $(f(K\tilde u)-K\tilde u \cdot u)/K$
converges to zero as $K$ goes to $\pm$infinity for every $\tilde u \in \R^d$ (see
\cite{BK,BS,FS}).

Fixing a direction $u$, we construct a harmonic surface $h$ with tilt $u$ as the
limit (and a fixed point) of a stochastic process introduced by Hammersley
called the harness process, \cite{PabloNeader,Ham}. The process is easily described by
associating to each point $s$ of $\S$ a one-dimensional homogeneous Poisson
process of rate 1. Fix an initial surface $\eta_0$ and for each point $s$ at the
epochs $\tau$ of the Poisson process associated to $s$ update $\eta_\tau(s)$ to
the average of the heights $\{\eta_{\tau-}(s'),\, s'$ is a neighbor of $s\}$. It
is clear that if $h$ is harmonic, then $h$ is invariant for this dynamics. We
start the harness process with $\eta_0=\gamma$, the hyperplane defined by
$\gamma(s) = s_i$, the $i$-th coordinate of $s$ and show that
$\eta_t(\cdot)-\eta_t(0)$ converges to $h$ in $L_2(\Pe\times P)$, where $\Pe$ is
the law of the point configuration $\S$ and $P$ is the law of the dynamics.

We prove that the tilt is invariant for the harness process for
each $t$ and in the limit when $t\to\infty$. In a finite graph the
average of the square of the height differences of neighbors is
decreasing with time for the harness process. Since essentially the
same happens in infinite volume, the gradients of the surface converge
under the harness dynamics. It remains to show that: (1) the limit of
the gradients is a gradient field and (2) the limit is harmonic. Both
statements follow from almost sure convergence along subsequences.

A key ingredient of the approach is the expression of the tilt of a
surface as the scalar product of the gradient of the surface with a specific
field (see Section \ref{sec:tilt}). This implies that the limiting
surface has the same tilt as the initial one.

\section{Preliminaries and main result}

\paragraph{Point processes and harmonic surfaces} Let $S$ be an ergodic point
process on $\R^d$ with intensity 1; call $\Pe$ its law and $\E$ the associated
expectation.  The process $S$ takes values in $\Nlf$, the space of locally
finite point configurations of $\R^d$; we use the notation $\ss$ for point
configurations in $\Nlf$ and $S$ for random point processes in $\Nlf$. The
elements $s$ of $\ss$ are called \emph{points} and the elements $x$ of $\R^d$
are called \emph{sites}. In the same way we use $\Nlf^\circ$ for the space of
configurations in $\Nlf$ with a point at the origin and $\sso$ for point
configurations in that space. Let $\S$ denote the Palm version of $S$. We can
think
of $\S$ as $S$ conditioned to have a point in the origin. If $S$ is Poisson,
then $\S=S\cup\{0\}$. We abuse the notation and use $\Pe$ and $\E$ to denote the
law of $\S$ and its associated expectation. For $\ss\in\Nlf$ let the
\emph{Voronoi cell} of $s\in \ss$ be defined by $\Vor(s)=\{x\in\R^d:\, |x-s|\le
|x-s'|,$ for all $s'\in \ss\setminus \{s\}\}$. If the intersection of the
Voronoi cells of $s$ and $s'$ is a $(d-1)$-dimensional surface, we say that $s$
and $s'$ are \emph{Voronoi neighbors}. We consider the random graph with
vertices $\ss$ and edges $\{(s,s')\,:\,s$ and $s'$ are Voronoi neighbors in
$\ss\}$. If $\S$ is the Palm version of a Poisson process, the graph is a
triangulation a.s.\/ called the \emph{Delaunay triangulation} of~$\S$.  To a
site $x\in\R^d$ we associate the \emph{center} of the Voronoi cell containing
$x$: $\Cen(x)=\Cen(x,\ss)=s\in \ss$ if $x\in\Vor(s)$; if $x$ belongs to the
Voronoi cell of more than one point, use lexicographic order of the coordinates
(or any other rule) to decide who is the center. 
Let
\begin{eqnarray}
  \label{x1}
  \Xi_1&:=&\{(s,\ss)\in\R^d\times\Nlf: s\in \ss\}\nonumber\\
\Xi_2&:=&\{(s,s',\ss)\in\R^d\times\R^d\times\Nlf: s,s'\in \ss\}.\nonumber
\end{eqnarray}
Functions $\eta:\Xi_1\to\R$ are called \emph{surfaces} and functions
$\zeta:\Xi_2\to\R$ are called \emph{fields}. Denote by $\tau_x$ the translation
operator: for $x$ in $\R^d$, $\tau_x\ss:=\{s-x: s\in \ss\}$. If
$\eta(s,\ss)=\eta(0,\tau_s\ss)$ for every $s\in\ss$ we say that $\eta$ is a
\emph{translation invariant} surface. A field $\zeta$ is \emph{covariant} if
$\zeta(s'-s,s''-s,\tau_s\ss) = \zeta(s',s'',\ss)$ for all $s,s',s''\in \ss$. A
field $\zeta$ is a \emph{flux} if $\zeta(s,s',\ss)=-\zeta(s',s,\ss)$ for all $s,
s' \in \ss$.  The \emph{conductances} induced by $\ss$ is the field $a$ defined
by
\begin{equation}
  \label{b78}
 a(s,s',\ss):= \one\{s \hbox{ and $s'$ are Voronoi neighbors in
}\ss\}. 
\end{equation}
The \emph{Laplacian} operator is defined on surfaces $\eta$ by
\begin{equation}
  \label{b1}
  \Delta \eta(s,\ss) = \sum_{s'\in \ss} a(s,s',\ss) [\eta(s',\ss) - \eta(s,\ss)]
\end{equation}
 The
\emph{gradient} of a surface $\eta$ is the field defined by 
\[
\nabla \eta(s,s',\ss) = a(s,s',\ss) [\eta(s',\ss) - \eta(s,\ss)].
\]
For fields $\zeta\colon\Xi_2\to \R$ the \emph{divergence} $\Div
\zeta\colon\Xi_1 \to \R$ is given by
\[
\Div \zeta(s,\ss) = \sum_{s'\in \ss}a(s,s',\ss)\zeta(s,s',\ss).
\]
Hence $\Delta\eta = \Div \nabla \eta$. To simplify notation we may drop the
dependence on the point configuration when it is clear from the context. The
Laplacian, gradient and divergence depend on the conductances, but we drop this
dependence in the notation, as they are fixed by \eqref{b78} along the paper.

A surface $h$ is called \emph{harmonic} for $\ss\in\Nlf$ if $\Delta
h(s,\ss)= 0$ for all $s\in\ss$.

\paragraph{Pointwise tilt} We say that for $\ss\in\Nlf$ a surface $\eta$ has
\emph{tilt} $\cI(\eta,\ss)=(\cI_{e_1}(\eta,\ss),\ldots,\cI_{e_d}(\eta,\ss))$ if
for each $u\in\{e_1,\dots,e_d\}$ the following limits for $K\to\pm\infty$ exist,
coincide and do not depend on~$x\in \R^d$
\begin{equation}
  \label{eq:tilt}
  \cI_u (\eta,\ss)
  := \lim_{K\to\pm \infty}
  \frac{\eta(\Cen(x+Ku),\ss)-\eta(\Cen(x),\ss)}{K}.
\end{equation}

\paragraph{Harness process} Given a surface $\eta$, let $M_s\eta$ be the surface
obtained by substituting the height $\eta(s)$ with the average of the heights at
the neighbors of $s$:
\begin{eqnarray}
  \label{a1}
  (M_s\eta)(s') =
  \begin{cases}
\displaystyle    \frac1{a(s)}\sum_{s'\in \ss} a(s,s') \eta(s')&\hbox{if } s'= s,\\
    \eta(s') \qquad\hbox {if }s'\neq s\,,
 \end{cases}
\end{eqnarray}
where $a(s) = \sum_{s'\in S}a(s,s')$. Take a point configuration $\ss$
and define the  generator
\begin{eqnarray}
  \label{generator}
  L_\ss f(\eta) = \sum_{s\in \ss} [f(M_s\eta)-f(\eta)].
\end{eqnarray}
That is, at rate 1, the surface height at $s$ is updated to the average of the
heights at the neighbors of~$s$.  We construct this process as a function of a
family of independent one-dimensional Poisson processes $T=(T_n,\, n=1,2\dots)$
with law $P$. Take an arbitrary enumeration of the points, $\ss=(s_n, n\ge 1)$
(for instance, $s_n$ may be the $n$-th closest point to the origin) and update
the surface at $s_n$ at the epochs of $\tT_n$. When the point configuration is
random, say $\S$, ask $T$ to be independent of $\S$ and define the process as
above to obtain a process $(\eta_t,t\ge 0)$ as a function of $(\S, T)$, with the
product law $\Pe\times P$, and $\eta_0$. The resulting \emph{noiseless harness
  process} is Markov on the space of surfaces with generator $L_\S$. See Section
\ref{harness} for a rigurous construction.

\paragraph{Assumptions} We assume that $\Se$ is a stationary point process in
$\R^d$ with Palm
version $\S$, satisfying the following:
\begin{enumerate}
\item[{A1}] The law of $\Se$ is mixing. 
\item[{A2}] For every ball $B\subset R^d$, $\Pe(|S\cap\partial B| < d+2)=1$.
\item[{A3}] $\E\exp(\beta a(0,\S))<\infty$ for some positive constant
  $\beta$. The number of neighbors of the origin has a finite positive
  exponential moment.
\item[{A4}] $\E[(\ell_{d-1}(\partial\Vor(0,\S)))^2]<\infty$. The $d-1$
  Lebesgue measure of the boundary of the Voronoi cell of the origin has finite
  second moment. 
\item[{A5}] $\E[\sum_{\vv\in\S}a(0,\vv)|\vv|^r]<\infty$ for some $r>4$.

\item[{A6}] $\Pe(S \mbox{ is periodic})=0$.
\end{enumerate}

All these assumptions are satisfied if $S$ is a homogeneous Poisson
process. Assumption {A1} guarantees ``one dimensional'' ergodicity as in
\eqref{eq:incparte1} later. Assumption {A2} is sufficient to define the Delaunay
triangulation. Notice that A4 implies that the volume of the Voronoi
cell of the origin has finite second moment:
$\E[(\ell_{d}(\Vor(0,\S)))^2]<\infty$.

Assumption {A6} is used on the one hand in the Appendix to identify the motion
of a random walk on the Delaunay triangulation with the motion of the
enviroment as seen from the walker. On the other hand ergodicity and
aperiodicity of the point process imply that there exist measurable functions
$s_n:\Nlf\to\R^d$ such that 
\begin{enumerate}
 \item[B1] $s_{-n}(\tau_{s_n}\S) = -s_n(\S)$, 
\item [B2] $\S=\{\vv_n(\S);\,n\in\Z\}$, and
\item [B3] $\tau_{\vv_n(\S)}\S$ has the same distribution as $\S$ for every
$n\in\Z$.
\end{enumerate}
This is used to extend the properties of $\S$ to $\tau_s\S$, for
all $s\in \S$. The point is that $\tau_s\S$ has the same law as $\S$ only if $s$
is correctly chosen as was shown in \cite{MR2044812, Peres} for Poisson
processes and by Timar \cite{Timar04} under the condition that $S$ is ergodic
and $\Pe$-a.s. aperiodic; see Heveling and Last \cite{Last05}.

\begin{Theorem}
\label{t1}
Let $\S$ be the Palm version of the stationary point process satisfying {\rm
  A1-A6} and let $\gamma$ be a surface with covariant gradient, tilt
$I(\gamma)\in\R^d$ and $\C(|\nabla \gamma|^r)<\infty$ for some $r>4$. Then: (a)
There exists a harmonic surface $h$ with $h(0,\S)=0$ and $\cI(h)=\cI(\gamma)$
$\Pe$-a.s. (b) if $\eta_t$ is the harness process with initial condition
$\gamma$, then,
\begin{equation}
  \label{bp1}
  \lim_{t\to\infty}\E E [\eta_t(s_n)-\eta_t(0) - h(s_n)]^2\;=\;0,
\end{equation}
for any $n\in\Z$, with $s_n$ as in {\rm B1-B3}.  (c) In dimensions $d=1$ and $d=2$,
$h$ is the only harmonic surface with covariant gradient and tilt $I(\gamma)$.
\end{Theorem}

Let $c\in\R^d$; the \emph{hyperplane} $\gamma(s,\S) = c\cdot s, \, s\in \S$
satisfies the hypotheses of the theorem with $I(\gamma)=c$. Items (a) and (b) of
the theorem say that a surface with tilt $c$ evolving along the harness process
and seen from the height at the origin converges in $L_2(\Pe\times P)$ to a
harmonic surface $h$ with the same tilt and with $h(0)=0$.

Let $H=(h_1,\dots,h_d)$, where $h_i$ is the harmonic surface obtained in Theorem
\ref{t1} for the tilt~$e_i$. The graph with vertices $H(\S)=(H(s),\,s\in \S)$
and conductances $\tilde a(H(s),H(s')):= a(s,s')$ is harmonic:
\begin{equation}
\label{b44}
  H(s) = \frac{1}{a(s)}\sum_{s'\in \S} a(s,s')\, H(s')
\end{equation}
that is, each point is in the barycenter of its neighbors in the neighborhood
structure induced by the Delaunay triangulation of $\S$. This graph, called the
\emph{harmonic deformation} of the Delaunay triangulation, does not coincide
with the Delaunay triangulation of $H(\S)$.

\paragraph{\it Random walks in random graphs and martingales.} 
Let $Y_t=Y_t^\S$, be the random walk on $\S$ which jumps from $s$ to $s'$ at
rate $a(s,s')$.  Since $H(\S)$ is harmonic, the random walk $H(Y_t)$ on $H(\S)$
is a martingale and it satisfies the conditions of the martingale central limit
theorem (Durrett, \cite[page 417]{Durrett}). So, the invariance principle holds
for $H(Y_t)$. The extension of the invariance principle from the walk $H(Y_t)$
to the walk $Y_t$ requires the sublinearity in $|s|$ of the corrector $\chi(s)=
H(s)-s$.

\paragraph{\it Corrector.} Mathieu and Piatnitski \cite{Mathieu} and Berger and
Biskup \cite{Biskup} construct the corrector for the graph induced by the
supercritical percolation cluster in $\Z^d$. Both papers prove sharp bounds on
the asymptotic behavior of the corrector and, as a consequence, the quenched
invariance principle for $Y_t$ for every dimension $d\ge 2$. Key ingredients in
those proofs are heat kernels estimates obtained by Barlow \cite{barlow} (in
\cite{Biskup} they are used just for $d\ge 3$). Sidoravicius and Sznitman
\cite{ss} also used the corrector to obtain the quenched invariant principle for
$d\ge 4$. Several papers obtain generalizations of similar results on subgraphs
of $\Z^d$ \cite{BarlowDeuschel,BiskupPrescott, Mathieu}.  Caputo, Faggionato and
Prescott \cite{CFP} use the corrector to prove a quenched invariance principle
for random walks on random graphs with vertices in an ergodic point process on
$\R^d$ and conductances governed by i.i.d.\ energy marks.

\paragraph{\it Uniqueness.} The uniqueness of (the gradients of) a harmonic
function with a given tilt has been proved by Biskup and Spohn \cite{BS} for the
graph with conductances associtated to the bonds of $\Z^d$ under ``ellipticity
conditions'' (see (5.1) and Section 5.2 in that paper) and by Biskup and
Prescott \cite{BiskupPrescott} in the bond percolation setting in $\Z^d$ using
``heat kernel'' estimates, see Section \ref{uniqueness} later.

We obtain harmonic surfaces as limits of the zero temperature harness
process. The tilt of a surface is obtained as a scalar product with a
specific field and it is invariant for the process. This allows us to show that
the harmonic limits have the same tilt as the initial surface.

The paper is organized as follows. In Section \ref{space} we give basic
definitions, define the space $\H$ of fields as a Hilbert space and show a
useful integration by parts formula. In Section \ref{sec:tilt} we show that the
coordinates of the tilt of a surface can be seen as the inner product of its
gradient with a specific field in $\H$. In Section \ref{construction} we
describe the Harris graphical construction of the Harness process. In Section
\ref{lemmas} we prove the main theorem. Section \ref{uniqueness} deals with the
uniqueness of the harmonic surface in $d=2$.  

\section{Point processes, fields and gradients}
\label{space}

Let $\Nlf=\Nlf(\R^d)$ be the set of all locally finite subsets of $\R^d$, that
is, for all $\ss\in\Nlf$, $|\ss\cap B|$, the number of points in $\ss\cap B$, is
finite for every bounded set $B\subset\R^d$. We consider the $\sigma$-algebra
$\B(\Nlf)$, the smallest $\sigma$-algebra containing the sets $
\{\ss\in\Nlf\colon |\ss\cap B|=k\}$, where $B$ is a bounded Borel set of
$\R^d$ and $k$ is a positive integer.

\paragraph{Ces\`aro means and the space $\H$.}
Let $\C$ be the measure in $\Xi_2$ defined on $\zeta\colon\Xi_2 \to \R$ by
\begin{equation}
\label{norm.H}
\C(\zeta) = \int_{\Xi_2} \zeta \, d\C
=\frac{1}{2}\E\sum_{\vv\in \S}a(0,\vv,\S)\zeta(0,\vv,\S)
\end{equation}
This measure is absolutely
  continuous with respect to the second order Campbell measure associated to
  $\Pe$ with density $Z(u,v,\ss)=a(u,v,\ss)\delta_0(u)$. The space
$\H:=L_2(\Xi_2, \R, \C)$ is Hilbert with inner product
$\C(\zeta\cdot\zeta')$, where the field $(\zeta\cdot\zeta')$ is defined by
$$(\zeta\cdot\zeta')(s,s',\S) = a(s,s',\S)\zeta(s,s',\S)\zeta'(s,s',\S).$$
If two fields $\zeta$ and $\zeta'$ coincide in the pairs $(0,s)$ for all $s$
neighbor of the origin, then their difference has zero $\C$-measure and hence a
field in $\H$ is characterized by its values at $((0,s),\,s$ neighbor of the
origin).  Define the equivalence relation $\zeta\sim\zeta'$ if and only if
$\zeta(0,s,\sso)=\zeta'(0,s,\sso)$, for all neighbor $s$ of the origin. Each
class of equivalence in $\H$ has a canonical covariant representant obtained by
$\zeta(s,s',\ss):=\zeta(0,s'-s,\tau_s \ss)$ for $s,s' \in \ss$. So hereafter,
when we refer to a field in $\H$, we assume that it is the covariant
representant.

The space $\H$ was previously considered by Mathieu and Piatnitski
\cite{Mathieu}  when $(\S, a)$ are given by
the infinite cluster for supercritical percolation in $\Z^d$. The Hilbert
structure of this space is useful to obtain weak convergence for the dynamics.

Define the Ces\`aro limit of a field $\zeta \colon \Xi_2 \to \R$ by
\begin{equation}
  \label{cesaro}
  C(\zeta):=\lim_{\Lambda \nearrow \R^d}\frac{1}{2|\Lambda|}
\sum_{\{s,s'\}\cap \Lambda \neq \emptyset}a(s,s',S)\zeta(s,s',S),
\end{equation}
where $\Lambda=\Lambda(K):= [-K,K]^d\subset \R^d$. Since $S$ is ergodic, the
Pointwise Ergodic Theorem \cite[pp. 318]{VereJones2} implies
$C(\zeta)=\C(\zeta)$, $\Pe$-a.s. Analogously, for translation invariant surfaces
$\eta$ we define its Ces\`aro mean $C(\eta)$ (with a slight
abuse of notation) and we have $C(\eta) = \C(\eta) = \E(\eta(0,\Sm))$.

\begin{Lemma} [Mass Transport Principle
  \cite{BenajaminiLyonsPeresSchramm,BenjaminiSchramm, Haggstrom, Peres}]
  \label{masstransport} Let $\zeta:\Xi_2\to\R$ be a covariant field such that either
  $\zeta $ is nonnegative or $\E\sum_{\vv\in\S}|\zeta (0,\vv,\S)|<\infty$. Then
\begin{equation}
\label{mtp1}
\E\sum_{\vv\in\S}\zeta (0,\vv,\S)  =\E\sum_{\vv\in\S}\zeta (\vv,0,\S).
\end{equation}
\end{Lemma}
\begin{proof} Let $s_n$ be the maps introduced in B1-B3. Use B2 and Fubini in
the first identity and covariance of $\zeta$ in the second one to obtain
\begin{align*} 
  &\E \sum_{\vv\in\S}\zeta (0,\vv,\S)
  \;=\;\sum_{n\in\Z}\E \zeta   (0,s_n(\S),\S)
  \;=\;\sum_{n\in\Z}\E \zeta (-s_{n}(\S),0,\tau_{s_n(\S)}\S) \\
  &\quad 
  \;=\;\sum_{n\in\Z}\E \zeta (s_{-n}(\tau_{s_n(\S)}\S),0,\tau_{s_n(\S)}\S)
  \;=\;\sum_{n\in\Z}\E \zeta (s_{-n}(\S),0,\S) 
  \;=\;\E \sum_{\vv\in\S}\zeta (s,0,\S),
\end{align*} 
where we used B1 in the third identity, B3 in the fourth one and Fubini and
B2 again in the fifth one. 
\end{proof}

\begin{Lemma}[Integration by parts formula]\label{intporpartes} Let $\zeta\in
  \H$ be a flux and $\phi$ be a translation invariant surface satisfying $\E
  [a(0)\phi^2(0)]<\infty$. Then
\begin{equation}\label{partes}
\C(\nabla \phi\cdot\zeta)=-\C(\phi\cdot\Div\zeta).
\end{equation}
\end{Lemma}
\begin{proof}
Note that
\begin{align*}
  \C(\nabla \phi\cdot\zeta)
  &=\frac12\E \sum_{\vv\in\S}a(0,\vv,\S)\nabla \phi(0,\vv,\S)\zeta(0,\vv,\S)  \\
  &=\frac12\E \sum_{\vv\in\S}a(0,\vv,\S)\phi(\vv,\Sm)\zeta(0,\vv,\S)
  -\frac12\E \sum_{\vv\in\S}a(0,\vv,\S)\phi(0,\S)\zeta(0,\vv,\S)  \\
  &=\frac12\E \sum_{\vv\in\S}a(0,\vv,\S)\phi(\vv,\S)\zeta(0,\vv,\S)
  -\frac12\E [\phi(0,\S)\Div\zeta(0,\S)].
\end{align*}
Since $\zeta$ and $a$ are covariant and $\phi$ is translation invariant,
$a(\vv,\vv',\S)\phi(\vv',\S)\zeta(\vv,\vv',\S)$ is covariant and Lemma
\ref{masstransport} implies
\begin{align*}
&\E \sum_{\vv\in\S}a(0,\vv,\S)\phi(\vv,\S)\zeta(0,\vv,\S) \; = \;
\E \sum_{\vv\in\S}a(\vv,0,\S)\phi(0,\S)\zeta(\vv,0,\S)  \\
&\qquad = \;-\E \sum_{\vv\in\S}a(0,\vv,\S)\phi(0,\S)\zeta(0,\vv,\S) 
 \; = \;-\E [\phi(0,\S)\Div\zeta(0,\S)]  .
\end{align*}
We used that $\zeta$ is a flux and $a$ is symmetric.
\end{proof}

\section{Tilt}

We define here the ``integrated tilt'' $\cJ(\eta)$ for surfaces $\eta$ with
covariant gradient $\nabla\eta\in\H$.  The coordinates of $\cJ(\eta)$ are
defined as the inner product of the gradient field $\nabla\eta$ with a
conveniently chosen field. We then prove that the pointwise tilt $\cI(\eta,\S)$
coincides with $\cJ(\eta)$, $\Pe$-a.s.

Take a unit vector $u$ and a point configuration $\sso$.  For neighbors $s$ of
the origin, let $\bb(0,s,\sso)$ be the $(d-1)$-dimensional side in common of the
Voronoi cells of $0$ and $s$ and let $b_u(0,s,\sso)$ be the projection of
$\bb(0,s,\sso)$ over the hyperplane perpendicular to $u$, see Figure
\ref{zeta_u}. Define the field $\omega_u$ by
\begin{equation}
  \label{p12}
  \omega_u(0,s,\sso):= \sign(s\cdot
u)\,a(0,s,\sso)\,\ell_{d-1}(\bb_u(0,s,\sso)).
\end{equation}
\label{sec:tilt}
where $\ell_{d-1}$ is the $(d-1)$-dimensional Lebesgue measure. By assumption
A4, $\omega_u\in\H$ and since $\nabla\eta$ is also in $\H$, we can define
\begin{equation}
  \label{p13}
  \cJ_u(\eta) := \C(\nabla\eta \cdot\omega_u)\quad\hbox{ and }\quad 
\cJ(\eta) := (\cJ_{e_1}(\eta),\dots,\cJ_{e_d}(\eta)).
\end{equation}
\begin{center}
\begin{figure}
\psfrag{s}{\hspace{-3pt}$0$}
\psfrag{sprima}{\vspace{5pt}$s$}
\psfrag{ladovoronoi}{\hspace{0pt}\vspace{-12pt}$\bb(0,s,\sso)$}
\psfrag{proyvoronoi}{\hspace{-35pt}$\bb_u(0,s,\sso)$}
\centerline{\includegraphics[width=6cm]{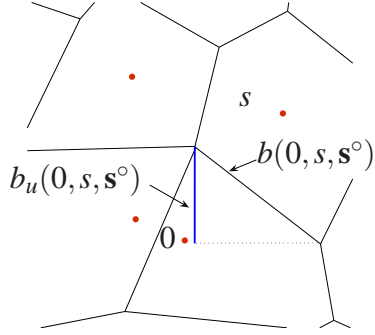}}
\caption{Definition of the field $\omega_u$ for $u=e_1$. }
\label{zeta_u}
\end{figure}
\end{center}
\begin{Proposition}
  \label{tilt}
  Let $\eta$ be a surface with covariant $\nabla\eta\in\H$. Then
\begin{equation}
  \label{p15}
\cI(\eta,\S) = \cJ(\eta),\qquad \Pe\hbox{-almost surely}.
\end{equation}
\end{Proposition}

Before proving the proposition we show a technical lemma.  Let $O_u$ be the
$d-1$ dimensional hyperplane orthogonal to $u$: $O_u=\{y\in \R^d\,:\,y\cdot u =
0\}$.

For $y\in O_u$ let $l_u(y)=\{y+\alpha\, u;\alpha\in\R\}$, the line containing
$y$ with direction $u$. Fix $\ss\in\Nlf$, define
$L_u(y,\ss):=\{\vv\in\ss\colon\Vor(\vv)\cap l_u(y)\neq\emptyset\}$, the set of
centers of the Voronoi cells intersecting $l_u(y)$. Define
$w:\R^{d}\times\Xi_2\to\{0,1\}$ by
\[w(y;\vv,\ww,\ss)=\begin{cases}
1&\qquad\mbox{ if } \bb(\vv,\ww,\ss)\cap l_u(y)\neq\emptyset;\\
0&\qquad\mbox{ otherwise}
\end{cases},\] the indicator that $\vv$ and $\ww$ are neighbors and its
boundary intersects the line $l_u(y)$. Define also $\theta:\R^{d}\times\Xi_1\to\R$ by
\[
\theta(y;\vv,\ss)=\sum_{\ww\in \ss}\ww
a^+(\vv,\ww,\ss)w(y;\vv,\ww,\ss),
\]
where $a^+(\vv,\ww,\Se)=a(\vv,\ww,\Se)\one\{(\ww\cdot u)>(\vv\cdot u)\}$. In
words, for $s\in L_u(y,\ss)$, $\theta(y;\vv,\ss)$ is the neighbor of $\vv$ in
the direction $u$ such that their boundary intersects $l_u(y)$.

\psfrag{A}{\hspace{-30pt}$\theta(y;s,\Se)$}
\psfrag{B}{$s_K$}
\psfrag{K}{\hspace{-4pt}$K$}
\psfrag{V}{\vspace{51pt}$v$}
\begin{figure}
	\centering
	\includegraphics[width=4cm,angle=-90]{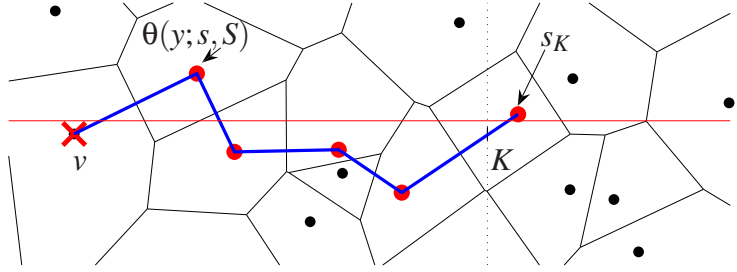}
	\caption{Points of $L(0,\ss)$ (red and big). The horizontal line is $l_y$.}
	\label{fig:tilt}
\end{figure}

For $x\in \R^d$, let $x^*\in O_u$ be the projection of $x$ over the
hyperplane $O_u$.  Observe that $w$ satisfies
\begin{equation}\label{eq:inc1}
  w(y;\vv,\ww,\ss)=w(y-x^*;\vv-x,\ww-x,\tau_x\ss),
\end{equation}
and 
\begin{equation}\label{eq:inc2}
\theta(y;\vv,\ss)-x=\theta(y-x^*;\vv-x,\tau_x\ss),
\end{equation}
for all $x\in\R^d$.

\begin{Lemma}\label{lemma1} Let $\zeta\in \H$ be a flux, $u$ a unit vector and
  $y\in \R^d$. Then
\begin{equation}\label{eq1}
  \E\sum_{\vv\in \Se}\zeta(\vv,\theta(y;\vv,\Se),\Se) 
  \one_{L_u(y,\Se)}(\vv)\one_{A}(\vv\cdot u)	
  =\ell_1(A)\C(\zeta\cdot\omega_{u})
\end{equation}
for all $A\in\B(\R)$ with 1-dimensional Lebesgue measure $\ell_1(A)<\infty$.
\end{Lemma}
The random set $\{(\vv\cdot u)\,:\,\vv\in L_u(y,\Se)\}$ is the one-dimensional
stationary point process obtained by projecting the points of $L_u(y,\Se)$ to
$l_u(y)$. One can think that each point $\vv$ has a weight
$\zeta(\vv,\theta(y;\vv,\Se),\Se)$. The expression on the left of \eqref{eq1} is
the average of these weights for the points projected over $A$. The expression
on the right of \eqref{eq1} says that this average contributes to the expression
as much as the Lebesgue measure of the projection over $O_u$ of the boundary
between $s$ and its neighbor in $L$ to its right.
\begin{proof} 
  By translation invariance we can take $y=0$ and, for simplicity we take
  $u=e_1$, the other directions are treated analogously.  In this case
  $O_u=\{x\in\R^ d\,:\,x_1=0\}$, $s\cdot u=s_1$, the first coordinate of $s$ and
  $x^*=(0,x_2,\dots,x_d)$.  Define
\[
g(\vv,\ss):=\zeta(\vv,\theta(0;\vv,\ss),\ss)\one_{L_u(0,\ss)}(\vv)\one_{A}(\vv_1).
\]
From the Generalized Campbell formula, \eqref{eq:inc2} and Fubini,
\begin{eqnarray}
  \E\sum_{\vv\in \Se}|g(\vv,\Se)|&=&\int_{\R^d}\E|g(x,\tau_{-x}\S)|dx\nonumber\\
&=&\int_{\R^d}\E|\zeta(0,\theta(-x^*;0,\S),\S)|
\one_{L_u(-x^*,\S)}(0)\one_{A}(x_1)dx\nonumber\\
&=&\ell_1(A)\int_{\R^{d-1}}\E\sum_{\vv\in \S}|\zeta(0,\vv,\S)|
\one_{\{\theta(x^*;0,\S)=\vv\}}\one_{L_u(x^*,\S)}(0)dx_2\dots dx_d\nonumber\\
&=&\ell_1(A)\E\sum_{\vv\in \S}|\zeta(0,\vv,\S)|\int_{\R^{d-1}}
\one_{\{\theta(x^*;0,\S)=\vv\}}\one_{L_u(x^*,\S)}(0)dx_2\dots dx_d\label{*1}
\end{eqnarray}
For $\vv\in \S$ such that $a^+(0,\vv,\S)=1$,
\[
\{\vv=\theta(x^*;0,\S),\ 0\in L_u(x^*,\S)\} \;=\;\{ l_u(x^*)\cap
b(0,\vv,\S)\neq\emptyset\}\; =\;\{x^*\in b_{u}(0,\vv,\S)\}.
\]
Hence, the integral in \eqref{*1} gives
$a^+(0,\vv,\S)\ell_{d-1}(b_u(0,\vv,\S))$ and
\begin{align*}
\E\Big|\sum_{\vv\in \Se}g(\vv,\Se)\Big|&=\ell_1(A)\E\sum_{\vv\in
\S}a^+(0,\vv,\S)|\zeta(0,\vv,\S)|\ell_{d-1}(b_u(0,\vv,\S))<\infty,
\end{align*}
because by Assumption A4 the field $\ell_{d-1}(b_u(0,\vv,\S))$ is in $\H$.
With the same computation, 
\begin{align}
\E\sum_{\vv\in \Se}g(\vv,\Se)&=\ell_1(A)\E\sum_{\vv\in
\S}a^+(0,\vv,\S)\zeta(0,\vv,\S)\ell_{d-1}(b_u(0,\vv,\S))\nonumber \\
&=\ell_1(A)\E\sum_{\vv\in
\S}a^+(0,\vv,\S)\zeta(0,\vv,\S)\omega_u(0,\vv,\S).\label{b20}
\end{align}
Since $ a^+(\vv,0,\S)\zeta(\vv,0,\S)\omega_u(\vv,0,\S)
=a^-(0,\vv,\S)\zeta(0,\vv,\S)\omega_u(0,\vv,\S) $, Lemma \ref{masstransport}
implies
\[
\E\sum_{\vv\in \Se}g(\vv,\Se)=\ell_1(A)\frac{1}{2}\E\sum_{\vv\in
  \S}a(0,\vv,\S)\zeta(0,\vv,\S)\omega_u(0,\vv,\S)=\ell_1(A)\C(\zeta\cdot\omega_u).
\qedhere
\]
\end{proof}

\begin{proof}[Proof of Proposition \ref{tilt}]
  Again without loosing generality we take $u=e_1$ and $u\cdot s=s_1$.
  Since $\nabla \eta$ is covariant and
  $\Cen(\cdot)$ is translation invariant (in the sense that $\Cen(x-z,\tau_z
  \ss) = \Cen(x,\ss)$),
\begin{equation}
\eta(\Cen(x+Ku,\Se),\Se)- \eta(\Cen(x,\Se),\Se) 
\;=\;\eta(\Cen(Ku,\tau_x\Se),\tau_x\Se)- \eta(\Cen(0,\tau_x\Se),\tau_x\Se).
\end{equation}
Since $\Se$ is stationary, if the limit in \eqref{eq:tilt} exists, it is
independent of $x\in\R^d$.  Let $K>0$, $A_K:=[0,K]\times\R^{d-1}$ and define
\[
\bar{\vv}_K={\rm argmax}\{(\vv\cdot e_1): \vv\in L(0,\Se)\cap A_K\}, 
\quad \vv_K=\theta(0;\bar{\vv}_K,\Se),
\]
that is, $\vv_K$ is the first point of $L(0,\Se)$ to the right of $A_K$.
Let $Z_K=\eta(\Cen(Ku,\Se),\Se)-\eta(\vv_K,\Se)$ and observe that (see Figure
\ref{fig:tilt})
\begin{equation}\label{b21}
  \eta(\Cen(Ku),\Se)- \eta(\Cen(0),\Se) 
  = \sum_{\vv\in L(0,\Se)\cap A_K} (\eta(\theta(0;\vv,\Se),\Se) - \eta(\vv,\Se))+Z_K
\end{equation}
Using \eqref{b21}, the limit as $K\to+\infty$ in \eqref{eq:tilt} reads
\begin{align}
  \cI_u(\eta,\Se) &= \lim_{K\to\infty}\frac{1}{K} \sum_{\vv\in L(0,\Se)\cap A_K}
  \nabla\eta(\vv,\theta(0;\vv,\Se),\Se))
  +\lim_{K\to\infty}\frac{Z_K}{K}\nonumber\\
  &= \lim_{K\to\infty}\frac{1}{K}\sum_{\vv\in \Se}
  \nabla\eta(\vv,\theta(0;\vv,\Se),\Se))\one_{L(0,\Se)}(\vv)
  \one_{[0,K]}(\vv_1)+\lim_{K\to\infty}\frac{Z_K}{K}.\label{p170}
\end{align}
Since $(Z_K)_{K\geq0}$ is a stationary sequence, the same arguments as in the
proof of Lemma \ref{lemma1} show that $\E|Z_0|\leq
\C(\Div|\nabla\eta|\ell(\Vor(0)))<\infty$, and so $Z_K/K\to 0$ almost surely as
$K\to\infty$. Then it suffices to show that the first limit in \eqref{p170}
converges to $\C(\nabla\eta\cdot\omega_u)$.  The sum in the first term in
\eqref{p170} can be telescoped as follows:
\begin{eqnarray}
\sum_{k=0}^{K-1}\sum_{\vv\in \Se} \nabla
\eta(\vv,\theta(0;\vv,\Se),\Se)\one_{L(0,\Se)}(s)\one_{[k,k+1]}(\vv_1)
&=&\sum_{k=0}^{K-1}\phi(\tau_{ku}\Se)
\end{eqnarray}
where $ \phi(\Se):=\sum_{s\in \Se}\nabla
\eta(s,\theta(0;s,\Se),S)\one_{L(0,\Se)}(s)\one_{[0,1]}(s_1).  $ Since the law
of $\Se$ is mixing, by Birkoff's ergodic theorem, for integer $K$,
\begin{equation}\label{eq:incparte1}
  \lim_{K\to\infty}\frac{1}{K}\sum_{k=0}^{K-1}\phi(\tau_{ku}\Se)\;=\;\E[\phi(\Se)]
  \;=\;\C(\nabla\eta\cdot\omega_u)\qquad\Pe\mbox{-a.s.},
\end{equation}
by Lemma \ref{lemma1}. 
If $K\in\R$, the result follows from the above and
\[
\lim_{K\to\infty}\frac{1}{K}|\eta(\Cen(Ku,\Se))-\eta(\Cen([K]u),\Se)|=0 
\quad \Pe\mbox{-a.s.}
\]
The same arguments work for $K<0$.
\end{proof}

\section{The harness process}
\label{harness}

\label{construction}

Given a configuration of points $\ss \in \Nlf$ we construct the
process $\eta_t^\gamma(\cdot,\ss): \ss\to\R$, with initial condition
given by a surface $\gamma(\cdot,\ss): \ss\to\R$ and generator given
by \eqref{generator}.

\paragraph{\it Graphical Construction}
Let $\tT=(\tT_n,\,n=1,2,\dots)$ be a family of independent Poisson processes of
intensity 1, $\tT_n\subset\R$. For fixed $\ss\in\Nlf$ and an arbitrary
enumeration of the points of $\ss=(s_1,s_2,\dots)$ we use the epochs of $T_n$ to
update the heights at $s_n$ as follows.  Fix $t>0$ and define a family
$(B^s_{[t,u]};u\leq t,\vv\in\ss)$ of backward simple random walks on $\ss$
starting at $\vv\in \ss$ at time $t$ and jumping at the epochs in $\tT$ as
follows. Start with $B^s_{[t,t]}=s$; then, if $\tau\in T_n$ and at time $\tau+$
the walk is at $s_n$ (that is, $B^s_{[t,\tau+]}=s_n$) then the walk chooses
uniformly $s'$, one of the neighbors of $s_n$ with probability
$\frac{1}{|a(s_n,\ss)|}$ and jumps over it, setting $B^s_{[t,\tau]}=s'$. Those
jumps are performed with the aid of independent uniform in $[0,1]$ random
variables $U=(U^k_n, \ k,n\ge 1)$; the variable $U^k_n$ is used to perform the
$k$-th jump from $s_n$. Now consider a random set of points $\S$ with law $\Pe$
and assume $U$, $\tT$ and $\S$ independent. Call $\pP$ and $\eE$ the probability
and expectation induced by $(\tT$, $U$), let $\PP=\Pe\times P$ and call $\EE$
the expectation with respect to $\PP$. Denote
\[
p_t(\vv,\ww,\S,\tT)\;:=\;\PP(B^s_{[t,0]}=\ww\,|\,\S,\tT),
\]
the probability that $B^s_{[t,0]}=\ww$ conditioned on the sigma field generated
by $(\S,\tT)$.  Define $\eta_t^\gamma(\vv,\S,T)$ as the expectation of
$\gamma(B^s_{[t,0]})$ conditioned on the sigma-field generated by $(\S,\tT)$:
\begin{equation}\label{eq:harness}
  \eta_t^\gamma(\vv,\S,\tT)\;:=\;\sum_{\ww\in \S}p_t(\vv,\ww,\S,\tT)\gamma(\ww).
\end{equation}
The $\eta$ process has initial configuration
$\eta_0^\gamma(\vv,\S,\tT)=\gamma(\vv)$ and evolves as follows. If $s=s_n$ and
$\tau\in \tT_{n}$ is an epoch of $\tT_n$, then
\[
p_\tau(\vv,\ww,\S,\tT)
=\sum_{\vv''\in\S}\frac{a(\vv,\vv'',\S)}{a(\vv,\S)}p_{\tau-}(\vv'',\ww,\S,\tT),
\]
and $\eta_\tau^\gamma(\vv,\S,\tT)$ is updated by
\begin{eqnarray}
\label{generatorr}
  \eta_\tau^\gamma(\vv,\S,\tT)
&=&\sum_{\ww\in \S}\sum_{\vv''\in\S}\frac{a(\vv,\vv'',\S)}{a(\vv,\S)}
p_{\tau-}(\vv'',\ww,\S,\tT)\gamma(\ww)\nonumber\\
&=&\sum_{\vv''\in\S}\frac{a(\vv,\vv'',\S)}{a(\vv,\S)}
\eta_{\tau-}(\vv'',\S,\tT)
\end{eqnarray}
while $\eta_\tau^\gamma(\ww,\S,\tT)$ remains unchanged for $\ww \ne \vv$. That
is, $\eta_\tau^\gamma(\cdot,\S,\tT)=M_{\vv}\eta_{\tau-}^\gamma(\cdot,\S,\tT)$.

\begin{Lemma}
  Given $\gamma:\Xi_1\to\R$ with $\nabla\gamma\in\H$, the process
  $\eta_t^\gamma(\cdot,\S,\cdot)$, is well defined $\Pe$-a.s. and has generator
  given by \eqref{generator}.
\end{Lemma}
\begin{proof} 
  To prove that the process is well defined we need to show that the sum on the
  right hand side of \eqref{eq:harness} is finite $\Pe$-a.s. Proposition
  \ref{prop:momentos-passeio} in the appendix shows that
\[
\EE|\eta_t^\gamma(0,\S,\tT)|\leq\EE\sum_{\ww\in
  \S}p_t(0,\ww,\S,\tT)|\gamma(\ww,\S)|\;\leq\; t\C(|\nabla\gamma|).
\]
This shows that the process is almost surely well defined at the
origin.  Using assumption A6b the result is extended to
all $\vv\in\S$.

The fact that $\eta_t^\gamma(\cdot,\S,\cdot)$ has generator given by
\eqref{generator}, follows from \eqref{generatorr} since $\S$ is locally finite
$\Pe$-a.s.
\end{proof}

We have constructed the process $\eta_t^\gamma$ as a deterministic function of
$\S$ and $\tT$, the point configuration plus the time epochs associated to the
points. That is, $\eta_t^\gamma$ is a random surface. Let
$(\S,\tT)=((\vv_n,\tT_n),\,n\ge 1)$ and $\tau_s (\S,\tT)=
((\vv_n-s,\tT_n),\,n\ge 1)$, for $\vv\in\S$.  Since
$p_t(\vv,\ww,(\S,\tT))=p_t(0,\ww-\vv,\tau_\vv(\S,\tT))$, $\gamma(0,\S)=0$ and
$\nabla\gamma$ is covariant,
\begin{align} 
  \eta_t^\gamma(\vv,(\S,\tT))
  &=\sum_{\ww\in \S}p_t(\vv,\ww,(\S,\tT))\,\gamma(\ww,\S)\nonumber\\
  &=\; \sum_{\ww\in\S}p_t(0,\ww-\vv,\tau_\vv(\S,\tT))\,\gamma(\ww-\vv,\tau_\vv\S)
  \,+\,\gamma(\vv,\S)\nonumber\\
  &=\;\sum_{\ww\in\tau_\vv{\S}}p_t(0,\ww,\tau_\vv(\S,\tT))\,
  \gamma(\ww,\tau_\vv\S)\,+\,\gamma(\vv,\S)
  \;=\;\eta_t^\gamma(0,\tau_\vv(\S,\tT))\,+\,\gamma(\vv,\S).\nonumber\label{eq:invariante}
\end{align}
If we call
\begin{equation}
  \label{h77}
  \psi_t(s,(\S,\tT)) := \eta_t^\gamma(0,\tau_s(\S,\tT)),
\end{equation}
then the process at time $t$ is the sum of the translation invariant surface $\psi_t$
and the initial condition $\gamma$. That is, 
\begin{equation}
  \label{h81}
  \eta_t = \psi_t + \gamma
\end{equation}
 In particular, it follows that $\nabla\eta_t^\gamma$
is a covariant (random) field $\PP$-a.s.

The dependence of $\eta_t^\gamma$ on $(\S,\tT)$ will be dropped from
the notation when clear from the context.

\paragraph{\bf Extension of the Hilbert space $\H$ to include the randomness
  coming from the process} We consider the probabilistic space where $\S,T,U$
are defined as independent processes and abuse notation by calling $\C$ the
Campbell measure on $\Xi_2$ associated to $\S,T,U$:
\[
\C(\nabla\eta_t^\gamma)\;:=\; \E E( \nabla\eta_t^\gamma)\;=\; \EE( \nabla\eta_t^\gamma)
\]
The following bound --shown in the Appendix-- implies that the process is well
defined as an element in $\H$ for all time.
\begin{Lemma}\label{momentosprocesso} If $\C(|\nabla\gamma|^r)<\infty$ then
\[
\C(|\nabla\psi_t^\gamma|^r)\;\le\;
2^r\C(|\nabla\gamma|^r)m^r(t)<\infty,
\]
where $m^r(t)$ denotes the $r-th$ moment of a Poisson random variable with mean $t$.
\end{Lemma}

As a consequence of \eqref{h81} and Lemma \ref{momentosprocesso} the tilt is
invariant under the dynamics:

\begin{Proposition}\label{tilt2}
  For all unitary $u\in\R^d$ and covariant surface $\gamma$ with
  $\nabla\gamma\in\H$, \[\cJ_u(\eta_t^\gamma)=\cJ_u(\gamma).\] for all $t\geq 0$.
\end{Proposition}
\begin{proof} 
  First observe that with $\Pe$-probability one we have,
\[
\Div \, {\omega_u}(0,\S) =\sum_{\vv\in \S} \omega_u(0,\vv,\S)
=\frac12\sum_{\substack{\vv\in 
    \S\\ (\vv \cdot
    \uu)>0}}\ell_{d-1}(b_u(0,\vv,\S))-\frac12\sum_{\substack{\vv\in \S\\ 
    (\vv \cdot \uu)<0}}\ell_{d-1}(b_u(0,\vv,\S))\;=\; 0,
\]
because each term in the substraction is the $(d-1)$-dimensional
Lebesgue measure of the projection of the Voronoi cell of the origin
over the hyperplane orthogonal to $u$. Then, 
\begin{align*}
  \cJ_u(\eta_t^\gamma)&=\;\C(\nabla \eta_t^\gamma\cdot\omega_u)
  \;=\;\C(\nabla\gamma\cdot\omega_u)+\C(\nabla\psi_t\cdot\omega_u)\\
  &=\C(\nabla\gamma\cdot\omega_u)-\C(\psi_t\cdot\Div{\omega_u})\;
  =\;\C(\nabla\gamma\cdot\omega_u) \;=\;\cJ_u(\gamma).
\end{align*}
where we used \eqref{h81} in the second identity, the integration-by-parts Lemma
\ref{intporpartes} in the third identity as $\psi_t^\gamma$ is a translation
invariant surface and $\Div{\omega_u}=0$ in the fourth identity.
\end{proof}

\section{The process converges to a harmonic surface}
\label{lemmas}
In this section we show that if $\gamma$ is a surface with tilt
$\cI(\gamma)$, whose gradient is in $\H$ and has more than 4 moments, then
there exists a surface $h$ with $\nabla h\in\H$ such that
$\nabla\eta_t^\gamma$ converges strongly in $\H$ to $\nabla
h$. Furthermore $h$ is harmonic and has the same tilt as
$\gamma$. We split the proof into several lemmas.
\begin{Lemma}\label{campoemh}
If $\C(|\nabla\gamma|^r)<\infty$ for some $r>4$, then for all $t>0$
\begin{equation}
  \label{dd11}
 \frac{d}{dt}\C({|\nabla\eta_t^{\gamma}|}^2)\;=\;-2\EE\left[
   a(0)^{-1}\left|\Delta\eta_t^{\gamma}(0,\Sm,T)\right|^2\right].  
\end{equation}
\end{Lemma}
\begin{proof}
  We drop the dependence on the initial condition $\gamma$, $\Sm$ and
  $T$ and write $\eta_t=\eta_t^\gamma(\cdot, \Sm, T)$. Let $
  \T_2=\bigcup_{\vv_n\in V_2}\tT_n, $ the epochs corresponding to
  sites in $V_2$, the set of second neighbors of the origin. Define
  the events
\begin{eqnarray*}
&F_1\;:=\;F_1(t,h)\;=\;\{|\T_2\cap[t,t+h]|=1\};\\
&F_{1,\vv}\;:=\;F_{1,\vv}(t,t+h)\;=\;F_1\cap\{|\T(\vv)\cap[t,t+h]|=1\}\cap\{\vv\in V_2\};\\
&F_2\;:=\;F_2(t,h)\;=\;\{|\T_2\cap[t,t+h]|\geq2\}.
\end{eqnarray*}
 Given $S$, $\T_2$ is a Poisson process with intensity $|V_2|$, hence
\begin{align}
&\PP(F_1|\S)\;=\;\EE[\one_{F_1}|\S]\;=\;|V_2|he^{-|V_2|h},\label{eq:f1}\\
&\PP({F_{1,\vv}}|\S)\;=\;he^{-|V_2|h}\one_{V_2}(\vv),\label{eq:f1v}\\
&\PP(F_2|\S) \;\leq\; h^2|V_2|^2. \label{eq:f2}
\end{align}
We have to compute
\begin{equation}
\EE \sum_{\vv\in\S}a(0,\vv)(|\nabla\eta_{t+h}(0,\vv)|^2-|\nabla\eta_t(0,\vv)|^2)(\one_{F_1} + \one_{F_2}) \;=\; I + II \label{eq:deriv_partes}
\end{equation}
We use
\[
|\nabla\eta_{t+h}(0,\vv)|^2-|\nabla\eta_t(0,\vv)|^2
\;=\;[\nabla\eta_{t+h}(0,\vv)-\nabla\eta_t(0,\vv)]^2
+2\nabla\eta_t(0,\vv)[\nabla\eta_{t+h}(0,\vv)-\nabla\eta_t(0,\vv)],
\]
\[
\Delta^\star\eta(\vv) \;:=\;\frac{1}{|a(\vv)|}\sum_{\ww\in \S}a(\vv,\ww)
(\eta(\ww)-\eta(\vv)) \;=\;M_\vv\eta(\vv)-\eta(\vv)
\]
to compute each term in \eqref{eq:deriv_partes}. Assume $F_1$ occurs.

\begin{itemize}
\item If the mark is neither at the origin nor at a neighbor of it,
  then $a(0,\vv)=0$, $\nabla\eta_{t+h}(0,\vv)=\nabla\eta_t(0,\vv)$,
  and the difference is zero.
\item If the mark is at the origin and $a(0,\vv)=1$,
\begin{eqnarray}
|\nabla\eta_{t+h}(0,\vv)|^2-|\nabla\eta_t(0,\vv)|^2
&=&[-M_0\eta_t(0)+\eta_t(0)]^2+2\nabla\eta_t(0,\vv)[-M_0\eta_t(0)+\eta_t(0)]
 \nonumber\\  
&=&-2\nabla\eta_t(0,\vv)\Delta^\star\eta_t(0)+|\Delta^\star\eta_t(0)|^2.
\label{eq:deriv1}
\end{eqnarray}
\item If the mark is at some $\vv$ such that $a(0,\vv)=1$, we have
  $\nabla\eta_{t+h}(0,\ww)=\nabla\eta_t(0,\ww)$, for all $\ww\neq \vv$. So
\begin{eqnarray}
|\nabla\eta_{t+h}(0,\vv)|^2-|\nabla\eta_t(0,\vv)|^2
&=&[M_\vv\eta_t(\vv)-\eta_t(\vv)]^2+2\nabla\eta_t(0,\vv)[M_\vv\eta_t(\vv)-\eta_t(\vv)]
\nonumber\\
&=&2\nabla\eta_t(0,\vv)\Delta^\star\eta_t(\vv)+|\Delta^\star\eta_t(\vv)|^2
\label{eq:deriv2}.
\end{eqnarray}
\end{itemize}
Given $\S$, the process $\T_2\cap[t,t+h]$ is independent of $\eta_t$, so
conditioning on $\S$ by \eqref{eq:f1}, \eqref{eq:f1v}, \eqref{eq:deriv1} and
\eqref{eq:deriv2}, we get that the first term in \eqref{eq:deriv_partes} equals
\begin{align*}
  h \EE\Bigl( \e^{-|V_2|h}\sum_{\vv\in
    \S}a(0,\vv)(2\nabla\eta_t(0,\vv)\nabla\Delta^\star\eta_t(0,\vv)
  +|\Delta^\star\eta_t(\vv)|^2+|\Delta^\star\eta_t(0)|^2)\Bigr) .
\end{align*}
By monotone convergence,
\begin{align}
  &\lim_{h\to 0}\frac{1}{2h} \EE\Bigl( \sum_{\vv\in \S}a(0,\vv)(|\nabla\eta_{t+h}(0,\vv)|^2-|\nabla\eta_t(0,\vv)|^2)\one_{F_1}\Bigr)\nonumber\\
  &\hspace{2cm} =\;\EE \Bigl(\sum_{\vv \in
    \S}a(0,s)\nabla\eta_t(0,s)\nabla\Delta^\star\eta_t(0,s)\Bigr) +\frac12
  \EE\Bigl(\sum_{\vv\in \S}a(0,\vv)(|\Delta^\star\eta_t(\vv)|^2
  +|\Delta^\star\eta_t(0)|^2)\Bigr)\nonumber\\
  &\hspace{2cm}=\; \EE \Bigl(\sum_{\vv \in
    \S}a(0,s)\nabla\eta_t(0,s)\nabla\Delta^\star\eta_t(0,s)\Bigr) + \EE\Bigl(
  a(0)|\Delta^\star\eta_t|^2\Bigr),\label{eq:derivparte1}
\end{align}
by the Mass Transport Principle \eqref{mtp1}.  Let $1/p+1/q=1$ and
$\zeta(0,\vv):=\one_{F_2}$, then for any time $t'$, by means of \eqref{h81} and
Lemma \ref{momentosprocesso}, the second term in \eqref{eq:deriv_partes} reads
\begin{eqnarray}
  \frac1h\EE \Bigl(\sum_{\vv\in \S}|\nabla\eta_{t'}(0,\vv)|^2\one_{F_2}\Bigr)
  &=&\frac{2}{h}\C(|\nabla\eta_{t'}|^2\zeta)\nonumber \\
  &\le&\frac{2}{h}\C(|\nabla\eta_{t'}|^{2p})^{1/p}\C(\zeta^q)^{1/q}\nonumber \\
  &=&
  \frac1h\Bigl[\EE\Bigl(\sum_{\vv\in\S}|\nabla\eta_{t'}(0,\vv)|^{2p}\Bigr)\Bigr]^{1/p}
  \Bigl[\EE a(0)\one_{F_2}\Bigr]^{1/q}
  \nonumber \\
  &\leq& (A m^{2p}(t')+B)\Bigl[\E|V_2|^3\Bigr]^{1/q}h^{2/q-1},\nonumber 
\end{eqnarray}
for constants $A,B>0$, where $m^r(t)$ is the $r$-th moment of a Poisson random
variable with mean~$t$. Choosing $q<2$ and applying this bound for $t'=t$ and
$t'=t+h$ we get
\begin{equation}\label{eq:derivparte2}
  \lim_{h\to 0}\frac{1}{2h}\EE\Bigl(\sum_{\vv\in \S}a(0,\vv)(|\nabla\eta_{t+h}(0,\vv)|^2-|\nabla\eta_t(0,\vv)|^2)\one_{F_2}\Bigr)= 0.
\end{equation}
From \eqref{eq:deriv_partes}, \eqref{eq:derivparte1},
\eqref{eq:derivparte2} and the integration by parts formula we obtain
\[
\frac{d}{dt}\C(|\nabla\eta_t|^2)\;=\;2\C(\nabla\eta_t\nabla\Delta^\star\eta_t)+\EE[ a(0)|\Delta^\star\eta_t|^2]
\;=\;-\EE[ a(0)|\Delta^\star\eta_t|^2].\qedhere
\]
\end{proof}

\begin{cor}\label{exist}
  If $\gamma$ satisfies the hypotheses in Lemma \ref{campoemh}, then

  (a) $\C(|\nabla\eta_t^\gamma|^2)$ is non-increasing in~$t$; 

(b) $\C(|\nabla\eta_t^\gamma|^2)$  is strictly decreasing at time $t$ if and only if
  $\eta_{t}^\gamma$ is not harmonic for $(a,\S)$; 

  (c) $\lim_{t\to\infty} a(0)^{-1}\Delta\eta_t^\gamma(0)\;=\; 0$,
  $\PP$-a.s.\/ and in $L_2(P)$, $\Pe$-a.s.;

  (d) $\lim_{t\to\infty}\Delta\eta_t^\gamma \;=\; 0$ $\PP$-a.s.
\end{cor}

\begin{proof}
Let 
\[
Z_t:=\frac{|\Delta\eta_t^\gamma(0)|^2}{ a(0)}=
a(0)|\Delta^\star\eta_t^\gamma(0)|^2.\] Lemma \ref{campoemh} implies
$\int_0^\infty \EE[Z_t]dt<\infty$. Fix $t_0=0$ and denote $0=t_0<t_1<t_2<\dots$
the ordered epochs of the superposition of the Poisson processes associated to
the point at the origin and its neighbors. This is a Poisson process with
intensity $a(0)+1$. For each $n\geq 0$, given $\S$, $Z_{t_n}$ is independent of
$(t_{n+1}-t_n)$. Hence,
\begin{align*}
\int_0^\infty \EE Z_tdt&=\EE\int_0^\infty Z_tdt=\sum_{k=0}^\infty \EE [Z_{t_k}(t_{k+1}-t_k)]
=\sum_{k=0}^\infty \EE\Bigl(\frac{Z_{t_k}}{ a(0)+1}\Bigr) < \infty.
\end{align*}
Hence,
\[
\sum_{k=0}^\infty \Delta\eta_{t_k}^\gamma(0)<\infty\quad\mbox{ and }
\lim_{t\to\infty} \Delta\eta_t^\gamma(0)=0\quad\PP\mbox{-a.s.}
\]
The $L_2(P)$ convergence follows by dominate convergence using that
$\Delta\eta_{t}^\gamma(0)\leq\sum_{k=0}^\infty
\Delta\eta_{t_k}^\gamma(0)$.
\end{proof}

\begin{proof}[{\bf Proof of (a) and (b) of Theorem \ref{t1}.}]
  
  In the notation we drop the dependence on the initial surface $\gamma$.
We want to prove the existence of a harmonic
  surface $h$, with covariant $\nabla h$ and such that for all $n\in\Z$
\[
\lim_{t\to\infty}\EE a(0,\vv_n)[\nabla\eta_t(0,\vv_n)-\nabla
h(0,\vv_n)]^2=0.
\]
where $(s_n, n\in\Z)$ is the enumeration of $\S$ given in B1-B3.
 
Observe that
\begin{align*}
  \EE|a(0,\vv_n)(\nabla\eta_t(0,\vv_n)-\nabla
  h(0,\vv_n))|^2\;\leq\;\EE\hspace{-2.1pt}\sum_{\vv\in
    \S}a(0,s)|\nabla\eta_t(0,\vv)-\nabla h(0,\vv)|^2\; =\;2\C(|\nabla\eta_t-\nabla h|^2).
\end{align*}
So, it is enough to show that $\nabla\eta_t \to \nabla h$ strongly in
$\H$.  

\emph{Existence of the limit .} Since by Corollary \ref{exist},
$\C(|\nabla\eta_t|^2)$ is bounded, $\nabla\eta_t$ is weakly compact, and hence
for every sequence $\{t_k\}_{k\geq 0}$, there exists a subsequence
$\{t_{k_j}\}_{j\geq 0}$ and a field $\zeta_\infty\in \H$ such that
\begin{equation}\label{limitefraco}
\lim_{j\to\infty}\C(\nabla\eta_{t_{k_j}}\cdot\zeta)=\C(\zeta_\infty\cdot\zeta),\quad \mbox{for all } \zeta\in \H.
\end{equation}

\emph{Uniqueness of the limit.} Let $\{t_k\}_{k\geq 0}$ be a subsequence such
that
$\nabla\eta_{t_{k}} \rightharpoonup\zeta_\infty$.

By \eqref{h81},
\begin{align}
\label{h78}
  \C(|\zeta_\infty|^2)
  &\;=\;\lim_{k\to\infty}\C(\nabla\eta_{t_k}\cdot\zeta_\infty)
  \;=\;\C(\nabla\gamma\cdot\zeta_\infty)
  +\lim_{k\to\infty}\C(\nabla\psi_{t_k}\cdot\zeta_\infty).
\end{align}
where $\psi_t$ is defined in \eqref{h77}.  Integrating by parts and using Hölder,
\begin{align}
\label{h79}
  &|\C(\nabla\psi_{t_k}\cdot\zeta_\infty)|
  \;=\;\lim_{j\to\infty}|\C(\nabla\psi_{t_k}\cdot\nabla\eta_{t_j})|
  \;=\;\lim_{j\to\infty}|\C(\psi_{t_k}\cdot\Delta\eta_{t_j})|\nonumber\\
  &\qquad\;\leq\;\lim_{j\to\infty}\EE(a(0)|\psi_{t_k}|^2)^{1/2}\EE(a(0)^{-1}|\Delta\eta_{t_j}|^2)^{1/2}\;=\;0,
\end{align}
by Corollary \ref{exist}. Therefore,
\begin{equation}\label{normalimite}
\C(|\zeta_\infty|^2)=\C(\nabla\gamma\cdot\zeta_\infty).
\end{equation}
Let $\nabla \eta_{t_k} \rightharpoonup \zeta_\infty$ and $\nabla \eta_{t_j}
\rightharpoonup \zeta'_\infty$ subsequences converging to two weak limits
$\zeta_\infty$ and $\zeta'_\infty$. By
\eqref{limitefraco} and \eqref{h78},
\begin{align}
  &\C(\zeta_\infty\cdot\zeta'_\infty)
  \;=\;\lim_{k\to\infty}\C(\nabla\eta_{t_k}\cdot\zeta'_\infty)
  \;=\;\C(\nabla\gamma\cdot\zeta'_\infty)
  +\lim_{k\to\infty}\C(\nabla\psi_{t_k}\cdot\zeta'_\infty)
  \;=\;\C(|\zeta'_\infty|^2),\label{prodinterno}
\end{align}
by \eqref{h79} and \eqref{normalimite}. The same holds for $\zeta_\infty$ and
so
$\C(|\zeta'_\infty|^2)=\C(|\zeta_\infty|^2)=\C(\zeta_\infty\cdot\zeta'_\infty).$
This implies $\C(|\zeta_\infty-\zeta'_\infty|^2)=0,$ i.e. there is a unique
limit point.

\emph{Strong convergence.} By \eqref{h81} and integration by parts,
\begin{equation}
  \C(|\nabla\eta_{t}|^2)
  \;=\;\C(\nabla\gamma\nabla\eta_{t})
  +\C(\nabla\psi_{t}\nabla\eta_{t})
  \;=\;\C(\nabla\gamma\nabla\eta_{t})
  -\C(\psi_t\Delta\eta_{t}).\label{limitenorma1}
\end{equation}
From H\"older's inequality,
\begin{equation}
\label{j1}
  (\C(\psi_{t}\Delta\eta_{t}))^2
  \;\leq\;\EE( a(0)|\psi_{t}(0)|^2)\EE\Bigl(\frac{|\Delta\eta_{t}(0)|^2}{ a(0)}\Bigr).
\end{equation}
Since by Lemma \ref{campoemh} $\EE\frac{|\Delta\eta_{t}|^2}{ a(0)}$ is
integrable, there exists a subsequence $(t_k)_{k\geq 0}$ such that
\[
\lim_{k\to\infty}t_k\EE\Bigl(\frac{|\Delta\eta_{t_k}(0)|^2}{ a(0)}\Bigr)\;=\;0.
\]
From Lemma \ref{lema:seg_mom_disc} in the appendix,
\begin{equation}\label{limitenorma2} 
  \lim_{k\to\infty}\frac{\E E|\gamma(B^0_{[t_k,0]})|^2}{t_k}t_k
  \EE\frac{|\Delta\eta_{t_k}(0)|^2}{ a(0)}\;=\;0.
\end{equation}
Using \eqref{normalimite}, \eqref{limitenorma1} and
\eqref{limitenorma2}, 
$C(|\nabla\eta_{t}|^2)\to\C(|\zeta_\infty|^2),$ and hence
$\nabla\eta_{t}$ converges strongly in $\H$ to $\zeta_\infty$.

\emph{Zero divergence.} By Jensen's inequality and using $a(0) \ge 2$, we get
\begin{align*}
\lim_{t\to\infty}\EE(a(0)^{-2}|\Delta\eta_t-\Div{\zeta_\infty}|^2)
& \leq  \lim_{t\to\infty}\EE(a(0)^{-1}\sum_{s\in \S}
a(0,s)(\nabla\eta_t(0,s)-\zeta_\infty(0,s))^2 )\\
  & \leq
\lim_{t\to\infty}\C(|\nabla\eta_t-\zeta_{\infty}|^2)\;=\;0.
\end{align*}
It follows by Corollary \ref{exist} that
\begin{equation}\label{limiteharmonico}
\Div{\zeta_\infty}=0\quad\PP\mbox{-a.s.}
\end{equation}

\emph{Covariance.} A field $\zeta\in\H$ is characterized by its values on the
edges leaving the origin. Therefore, by taking the covariant canonical
representant defined by
$\zeta_\infty(\vv,\ww,\S):=\zeta_\infty(0,\ww-\vv,\tau_{\vv}\S)$, we can
consider $\zeta_\infty$ to be covariant.

\emph{Gradient field.} To show that $\zeta_\infty$ is a gradient field we prove
that it verifies the co-cycle property, that is there exists
$\Nlf^\star\subseteq\Nlf$, with $\Pe(\Sm\in\Nlf^\star)=1$ and such that for all
$\ss\in\Nlf^\star$ and every closed path
$\vv_{i_0},\vv_{i_1},\ldots,\vv_{i_k}=\vv_{i_0}\in\ss$ with
$a(\vv_{i_j},\vv_{i_{j-1}})=1$, $j=1,\ldots,k$ we have
$\sum_{j=1}^k\zeta_\infty(\vv_{i_j},\vv_{i_{j-1}},\ss)=0.$

Let $n,m\in\Z$. Since $a(s_n,s_m)\nabla\eta_t(s_n,s_m)$
$\overset{L_2(\PP)}{\to}$ $a(s_n,s_m)\zeta_\infty(s_n,s_m)$, we have a
subsequence that converges almost surely. Denote by $\Nlf_{n,m}\subset\Nlf$ the
set where convergence holds.  Using a standard diagonal argument we get a
subsequence $(t_k)_{k\geq0}$ such that \[a(s_n,s_m)\nabla\eta_{t_k}(s_n,s_m)
\overset{\mbox{a.s}}{\longrightarrow} a(s_n,s_m)\zeta_\infty(s_n,s_m) \quad
\mbox{ for all } n,m\in\Z.\] Define
$\Nlf^\star=\bigcap_{n,m\in\Z}\Nlf_{n,m}$. Since the co-cycle property holds for
every $t$ the a.s.\ convergence implies the co-cycle property for
$\zeta_\infty$.

\emph{Tilt.}  The tilt is a continuous functional in $\H$ and it is
constant for the dynamics by Proposition~\ref{tilt2}. Hence the limit
$\zeta_\infty$ has the same tilt as the initial surface. This completes the proof of
(a) and (b) of the theorem.
\end{proof}

\section{Uniqueness of harmonic  surfaces in $d=2$.}
\label{uniqueness}
In this section we prove uniqueness (up to an additive constant) of the harmonic
surface with covariant gradient for $d=2$. Observe that in dimension one the
harmonic function with a given tilt can be explicitly computed and hence the
uniquness follows immediately. To prove uniqueness for $d=2$ we use the
following result. 
\begin{Theorem}[Theorem 5.1 of Berger and Biskup \cite{Biskup}]
  \label{teo:max}
  For $c\in\R^2$, let $\gamma(s)=c\cdot s$, and $h$ be a harmonic surface for
  $a(\cdot, \cdot,\Sm)$ with covariant gradient in $\H$ and tilt
  $\cI(h)=\cI(\gamma)$. Then
\begin{equation}
\label{cota.unif}
\lim_{n\to\infty}\frac{1}{n}\max_{\vv\in \S\cap[-n,n]^2}\{|h(\vv)-c\cdot\vv|\}=0, \quad \Pe\mbox{-{\rm a.s.}}
\end{equation}
\end{Theorem}
We omit the proof; it follows \cite{Biskup}, details can be found in
\cite{tesis.rafa}. Berger and Biskup \cite{Biskup} use this theorem to show
uniqueness of the harmonic surface on the supercritical bond-percolation cluster
in $\Z^2$; we adapt their proof to our case. Theorem 2.4 of Bikup and Prescott
\cite{BiskupPrescott} proves \eqref{cota.unif} for bond percolation in $\Z^d$
for all $d\ge 2$ under ``heat kernel estimates'' assumptions, see (2.17) and
(2.18) in that paper.  These estimates are to be established in our setting.
  
\begin{proof}[{\bf Proof of (c) of Theorem \ref{t1}.}]
  It is enough to show that if $h$ is a harmonic surface with $\cI(h)=0$, then
  $\nabla h=0$ or, equivalently, $\C(|\nabla h|^2)=0$. From the considerations
  after \eqref{cesaro}, if $\nabla h\in\H$ then, with probability $1$,
\begin{eqnarray*}
\C(|\nabla
h|^2)&=&\lim_{n\to\infty}\frac{1}{2(2n)^2}\sum_{\vv\in\Se\cap[-n,n]^2}\sum_{
\ww\in\Se}a(\vv,\ww)|\nabla h(\vv,\ww)|^2
\end{eqnarray*}
Let $S_n =\Se\cap{[-n,n]^2}$. Using that $h$ is harmonic rewrite the sum at the right hand side as
\begin{align*}
\sum_{\substack{\vv\in S_n \\ \ww\in\Se}}a(\vv,\ww)|\nabla h(\vv,\ww)|^2
= &  \sum_{\substack{\vv\in S_n \\ \ww\in\Se}}a(\vv,\ww)h(\ww)\nabla h(\vv,\ww)
- \sum_{\vv\in S_n}h(\vv)\sum_{\ww\in\Se}a(\vv,\ww)\nabla h(\vv,\ww)
\\
= & \sum_{\substack{\vv\in S_n \\ \ww\in\Se}}a(\vv,\ww)h(\ww)\nabla
h(\vv,\ww).
\end{align*}
Using harmonicity again, we obtain
\begin{align*}
\sum_{\substack{\vv\in S_n \\ \ww\in\Se}}a(\vv,\ww)|\nabla h(\vv,\ww)|^2
= & \sum_{\substack{\vv\in S_n \\ \ww\in S_n}}a(\vv,\ww)h(\ww)\nabla
h(\vv,\ww) + \sum_{\substack{\vv\in S_n \\ \ww\in \Se\backslash
S_n}}a(\vv,\ww)h(\ww)\nabla
h(\vv,\ww).\\
= & \sum_{\substack{\vv\in S_n \\ \ww\in S_n}}a(\ww,\vv)h(\vv)\nabla
h(\ww,\vv) + \sum_{\substack{\vv\in S_n \\ \ww\in \Se\backslash
S_n}}a(\vv,\ww)h(\ww)\nabla
h(\vv,\ww).\\
= & -\sum_{\substack{\vv\in S_n \\ \ww\in \Se
\backslash S_n}}a(\ww,\vv)h(\vv)\nabla
h(\ww,\vv) + \sum_{\substack{\vv\in S_n \\ \ww\in \Se\backslash
S_n}}a(\vv,\ww)h(\ww)\nabla
h(\vv,\ww).\\
= & \sum_{\substack{\vv\in S_n \\ \ww\in\Se\backslash
S_n}}a(\vv,\ww)(h(\vv)+h(\ww))\nabla h(\vv,\ww)
\end{align*}
Then, with $\Pe$-probability $1$,
\begin{eqnarray*}
\C(|\nabla h|^2)&=&\lim_{n\to\infty}\frac{1}{8n^2}\sum_{\substack{\vv\in S_n \\
\ww\in\Se\backslash S_n}}a(\vv,\ww)(h(\ww)+h(\vv))\nabla h(\vv,\ww).
\end{eqnarray*}
Since this limit exists a.s., we are done if we can show that the r.h.s
converges to zero in probability. Observe that
\[
\Big| \hspace{-7pt}\sum_{\substack{\vv\in S_n \\
    \ww\in\Se\backslash S_n}}a(\vv,\ww)(h(\ww)+h(\vv))\nabla h(\vv,\ww)\Big|
\;\leq\; \max_{\substack{\vv\in S_{n},\\
    \ww\in\Se\backslash
    S_n}}\{a(\vv,\ww)|h(\vv)+h(\ww)|\}\hspace{-5pt}\sum_{\substack{\vv\in S_n \\
    \ww\in\Se\backslash S_n}}a(\vv,\ww)|\nabla h(\vv,\ww)|.
\]
Let $A_n:=\{\mbox{There exists $\vv\in\Se_n$ and $\ww\in\Se\backslash\Se_{2n}$ such
  that $a(\vv,\ww)=1$}\}$, and observe that
\[
\Pe(A_n)\;\leq\; \E\sum_{\vv\in \Se_n}\sum_{\ww\in\Se\backslash\Se_{2n}}a(\vv,\ww)
\;\leq\;\E\sum_{\vv\in
\Se_n}\sum_{\ww\in\Se\backslash\Se_{2n}}a(\vv,\ww)\frac{|\vv'-\vv|^4}{n^4}\;\leq\;
\frac{1}{n^2}\E[\sum_{\vv\in\S}a(0,\vv)|\vv|^4].
\]
Therefore, by Borel-Cantelli, the fact that $I(h)=0$ and Theorem \ref{teo:max},
given $\varepsilon$ we can take $n$ big enough such that
\[
{(2n)^{-1}}\max_{\vv\in S_{n},\ww\in\Se\backslash
S_n}\{a(\vv,\ww)|h(\vv)+h(\ww)|\}\leq\frac{1}{n}\max_{\vv\in
S_{2n}}\{|h(\vv)|\}<\varepsilon.
\]
It follows that
\[
\lim_{n\to\infty}{(2n)^{-1}}\max_{\vv\in S_{n},\ww\in\Se\backslash
S_n}\{a(\vv,\ww)|h(\vv)+h(\ww)|\}=0, \quad \Pe\mbox{-a.s.}
\]
and therefore it is enough to show that there exists a sequence $(Z_n)_{n\geq1}$
such that
\[
Z_n\geq\frac{1}{n}\phi_n(\Se):=\frac{1}{n}\sum_{\vv\in \Se_n}\sum_{\vv\in \Se\backslash\Se_n}a(\vv,\ww)|\nabla h(\vv,\ww)|,
\]
almost surely and $Z_n$ converges in probability.

Given $B,B'\in\B(\R^2)$, let $
\phi_{B,B'}(\Se):=\sum_{\vv\in\Se}\sum_{\ww\in\Se}a(\vv,\ww,\Se)|\nabla h(\vv,\ww,\Se)|\one_B(\vv)\one_{B'}(\ww),
$
and observe that by the refined Campbell formula and the covariance of $\nabla h$ and $a$
\begin{align}\label{eq:border1}
\E\phi_{B,B'}
&=\int_{\R^2}\E\sum_{\ww\in\tau_{-\vv}\Se}a(\vv,\ww,\tau_{-\vv}\Se)|\nabla
h(\vv,\ww,\tau_{-\vv}\Se)|\one_B(\vv)\one_{B'}(\ww)d\vv  \nonumber\\
&=\int_{\R^2}\E\left[\sum_{\ww\in\Se}a(\vv,\ww+\vv,\tau_{-\vv}\Se)|\nabla h(\vv,\ww+\vv,\tau_{-\vv}\Se)|\one_B(\vv)\one_{B'}(\ww+\vv)\right]d\vv \nonumber\\
&=\int_{\R^2}\E\left[\sum_{\ww\in\Se}a(0,\ww,\Se)|\nabla h(0,\ww,\Se)|\one_B(\vv)\one_{B'}(\ww+\vv)\right]d\vv \nonumber\\
&=\E\sum_{\ww\in\Se}a(0,\ww,\Se)|\nabla h(0,\ww,\Se)|\ell({B}\cap\tau_\ww {B'})
\end{align}

Let $B_n=[-n,n]^2$ and $\mathcal{X}_n$ be the family of half-planes defined
by the borders of $B_n$, and disjoint from $B_n$. It is clear that
\[
\sum_{\substack{\vv\in S_n \\ \ww\in\Se\backslash S_n}}a(\vv,\ww)|\nabla h(\vv,\ww)|\leq\sum_{B\in \mathcal{X}_n}\phi_{B_n,B}(\Se).
\]
We show the convergence of $\frac{1}{n}\phi_{B_n,B}(\Se)$ for a fixed $B\in
\mathcal{X}_n$. The convergence of the other terms follows from the same
arguments.

Before proceeding, we have yet another approximation to take care of. Let
$H_n=\R\times[n,+\infty)$, $G_n=[-n,n]\times(-\infty,n]$,
and observe that
\[
\phi_{B_n,H_n}(\Se)\leq\phi_{G_n,H_n}(\Se),\qquad\rm{a.s.}
\]

Let us see what happens with a fixed line first. To do that, let
$G=[0,1]\times\R^-$ and $G_n^o=[-n,n]\times\R^-$. If we define $T=\tau_{e_1}$,
by the covariance of $\nabla h$ and Birkhoff Ergodic Theorem, it follows that
\[
\lim_{n\to\infty}\frac{1}{n}\phi_{G_n^o,H_0}(\Se)=\frac{1}{n}\lim_{n\to\infty}\sum_{k=-n}^{n-1}\phi_{G,H_0}(T^k\Se)=2\E[\phi_{G,H_0}(\Se)]<\infty \quad \mbox{a.s.}
\]  
By the covariance of $\nabla h$ it follows that
\[
\lim_{n\to\infty}\Pe(|\phi_{G_n,H_n}(\Se)-2\E[\phi_{G,H_0}(\Se)]|>\epsilon n)=\lim_{n\to\infty}\Pe(|\phi_{G_n^o,H_0}(\tau_{ne_2}\Se)-2\E[\phi_{G,H_0}(\Se)]|>\epsilon n)=0,
\]
and the result follows.
\end{proof}

\section{Final Comments}

\subsection{Invariance Principle} The key ingredient to obtain an invariance
principle from the existence of a harmonic deformation of the original graph is
a uniform sublinear bound of the corrector as in \eqref{cota.unif}. Grisi
\cite{tesis.rafa} obtained this bound for the Poisson process following the
arguments of Berger and Biskup \cite{Biskup} in $d=2$. Hence the quenched
invariance principle holds in the Delaunay triangulation of a Poisson process.
Presumably this also holds for a non-periodic ergodic process satisfying
assumptions {A1-A7}. For $d\ge 3$ the
proofs of a quenched invariance principle in the percolation setting and
related models rely on heat kernel estimates like
those obtained by Barlow \cite{barlow}, which do not follow from the sublinear
behavior of the corrector along lines. An extension of these bounds to our case
are to be obtained.

\subsection{The process trajectory is orthogonal to the space of harmonic
  surfaces.} Since the tilt in the direction $u\in\R^d$ is a continuous
functional in $\H$, by Riesz Theorem, there exist a field $\omega_u \in \H$ such
that the tilt is given by the scalar product with $\omega_u$. In our
case, we have found explicitly that field (the one given in \eqref{p12}).

Given an initial condition $\gamma$, the process $\psi_t=\eta_t^\gamma - \gamma$
is a translation invariant surface and has zero tilt. The convergence of
$\nabla\psi_t$ follows from the convergence of $\nabla\eta_t^\gamma$, and the
limiting field is the gradient of the corrector $\nabla\chi_\gamma:=\nabla
h-\nabla\gamma$, for $h$ given by Theorem \ref{t1}. Integrating by
parts and using translation invariance, for $\zeta\in\H$ with
$\Div\zeta\equiv0$,
\[
\C(\nabla(\gamma-\eta_t^\gamma)\zeta) = - \C((\gamma - \eta_t^\gamma)\Div \zeta)
= 0,\qquad \hbox{for all }t\ge 0.
\] 
Hence $\gamma - \eta_t^\gamma$ is orthogonal to the subspace of fields in $\H$
with zero divergence (which contains the gradients of all harmonic
surfaces). In fact, $\nabla h$ is the orthogonal projection of
$\nabla\gamma$ over this subspace. In particular, we have
\begin{equation}
\label{decomposition}
\nabla \gamma= \nabla h + (\nabla\gamma - \nabla h)
\end{equation}

Mathieu and Piatnitski \cite{Mathieu} also consider $L_2(\Xi_2, \C)$. Equation
\eqref{decomposition} corresponds to their decomposition of the space as
$L_2(\Xi_2, \C) = L_2^{sol} \oplus L_2^{pot}$. Taking
$\gamma_i(s):=(e_i\cdot s)$, $i\leq d$, the surface
$\chi:=(\chi_{\gamma_1},\ldots,\chi_{\gamma_d})$ is what they call the {\em
  corrector}.

\subsection{Regularization effect.} The regularization effect observed in
Fig. \ref{fig:2s} can be explicitly formulated as follows.  If one takes $n$
arbitrary points $s_1,\ldots,s_n\in\R^2$, the barycenter minimizes the following
sum of scalar products
\begin{equation}
  \label{dd3}
  \arg\min_{x\in\R^2}\sum_{k=1}^n[(s_k-x)\cdot(s_{k+1}-x)]\;=\; \frac1n
\sum_{k=1}^ns_k.
\end{equation}
where $s_{n+1}=s_1$. Take a point configuration $\ss$ and $s,s'\in\ss$ neighbors
in the Delaunay triangulation of $\ss$. The
directed edge $(\vv,\ww)$ is shared by the triangles $\vv\ww\alpha_+$ and
$\vv\ww\alpha_-$, where $\alpha_+(\vv,\ww)$ is the first common neighbor of
$\vv$ and $\ww$ in the clockwise direction from $s'-s$ and $\alpha_-(\vv,\ww)$ is the other
common neighbor. 
We show the following extension of \eqref{dd3} to harmonic surfaces.
\begin{Lemma}
  \label{dd5}
Let $S$ be a stationary point process. Then the harmonic deformation of the
Delaunay triangulation of $S$ minimizes
\begin{equation}
  \label{dd15}
  \E \sum_{s \in \S} a(0,s)\big[G(s) \cdot G(\alpha_+(0,s))\big]
\end{equation}
among deformations $G:\S\mapsto \R^d$ of $\S$ such that $G(0)=0$ and the
corrector $G(s)-s$ has coordinates with gradient in $\H$.
\end{Lemma}
We prove this Lemma below. Given a surface $\eta$ define the
fields $\zeta_+^\eta, \zeta_-^\eta\colon \Xi_2\to\R$ by
\[\zeta_\pm^\eta(\vv,\ww)\;:=\;a(s,s')\nabla\eta(\vv,\alpha_\pm(\vv,\ww)).\] Any
two surfaces
$\eta, \phi\colon \Xi_1 \to \R$ satisfy
\begin{equation}
\label{dd1}
\C(\nabla\eta\zeta_+^\phi)=\C(\zeta_-^\eta\nabla\phi).
\end{equation}
Also note that
\begin{equation}
  \label{dd2}
  \sum_{\ww\in\S}a(\vv,\ww)\zeta_\pm^\eta(\vv,\ww)=\Delta \eta(\vv) \quad \mbox{and} \quad \sum_{\ww\in\S}a(\vv,\ww)\zeta_\pm^\eta(\ww,\vv)=0.
\end{equation}
If $\phi$ is a translation invariant surface (that is
$\phi(\vv,\ss)=\phi(0,\tau_\vv\ss)$) then, by the mass transport principle,
\begin{eqnarray}
  2\C(\nabla\phi\zeta_\pm^\eta)
  &=&\E\sum_{\vv\in\S}a(0,\vv)\nabla\phi(0,\vv)\zeta_\pm^\eta(0,\vv)\nonumber\\
&=&\E\, \phi(0)\sum_{\vv\in\S}a(0,\vv)\zeta_\pm^\eta(\vv,0)-\E\phi(0)\sum_{\vv\in\S}a(0,\vv)\zeta_\pm^\eta(0,\vv)\nonumber\\
&=&-2\C(\phi\Delta\eta) = 2\C(\nabla\phi\nabla\eta), \label{dd9}
\end{eqnarray}
where the first identity in the bottom line follows from \eqref{dd2} and the second one by the integration by parts formula. 
\begin{Lemma}
\label{dd13}
  \begin{equation}
    \label{dd10}
    \frac{d}{dt}\C(\nabla\eta_t\zeta_+^{\eta_t}) =
\frac12\frac{d}{dt}\C(|\nabla\eta_t|^2) = -\EE\left[
  a(0)^{-1}\left|\Delta\eta_t(0,\Sm)\right|^2\right].
  \end{equation}
\end{Lemma}
\begin{proof}
Using \eqref{dd1} and $\nabla
\eta_t = \nabla\gamma + \nabla\psi_t$,
\begin{eqnarray*}
\C(\nabla\eta_t\zeta_+^{\eta_t})&=&\C(\nabla\gamma\zeta_+^{\gamma})+\C(\zeta_-^\gamma\nabla\psi_t)+\C(\nabla\psi_t\zeta_+^{\eta_t})\\
&=&\C(\nabla\gamma\zeta_+^{\gamma})+\C(\nabla\gamma\nabla\psi_t)+\C(\nabla\psi_t\nabla\eta_t)\\
&=&\C(\nabla\gamma\zeta_+^{\gamma})+\C(\nabla\gamma\nabla\eta_t)+\C(\nabla\psi_t\nabla\eta_t)-\C(|\nabla\gamma|^2)\\
&=&\C(\nabla\gamma\zeta_+^{\gamma})+\C(|\nabla\eta_t|^2)-\C(|\nabla\gamma|^2),
\end{eqnarray*}
where the second identity follows from \eqref{dd9}. This shows the first
identity in \eqref{dd10}; the second identity is \eqref{dd11}.
\end{proof}

\begin{proof}[Proof of Lemma \ref{dd5}]
  Lemma \ref{dd13} shows that $\C(\nabla\eta_t\zeta_+^{\eta_t})$ is
  non-increasing, and that it is strictly decreasing if and only if $\eta_t$ is
  not harmonic and hence 
\[
\C(\nabla g\zeta_+^{g})=0 \quad \hbox{ if and only if } g \hbox{ is harmonic}
\]
Taking $g_i$ as the coordinates of $G$ and using that $G(0)=0$, we get
\eqref{dd15}.
\end{proof}

\subsection{Some simulations.} The first two pictures in Figure \ref{dibus} show
level curves of a linear interpolation of the surface $\gamma - h$. In the first
one some level curves are drawn. From blue (minimum) to red (maximum). The
level curve of zero is drawn in green. In the second one the sublevel set of
zero is drawn in blue and the superlevel set is drawn in red. The black curve
is the level set of zero. 

The next picture is the Voronoi tesselation of the harmonic points. The Delaunay
triangulation of this points does not necessarily coincide with the harmonic
deformation of the original Delaunay triangulation. It is easy to construct
examples where this in fact happens, and it can be seen in simulations. However,
it can be appreciated in simulations that the density of triangles in the
harmonic graph that are not Delaunay triangles is very low, as shown in Figure
\ref{dibus2}. Finally, on the bottom-right of Figure \ref{dibus}, the
level curves of the harmonic surface
with tilt $(1,0)$ is shown, that is the limit of the dynamics
with initial condition given by the hyperplane $\gamma(x,y)=x$. Observe
that the surface is pretty close to the original condition.

\begin{figure}
\[
\hspace{-60pt}\begin{array}{cc}
\includegraphics[height=6.2cm]{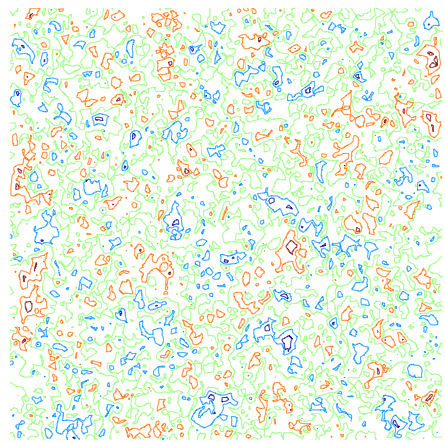}& \hspace{-40pt}
\includegraphics[height=6cm]{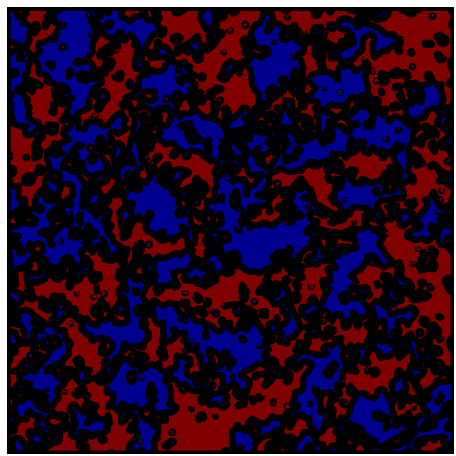}
\\
\includegraphics[width=6cm]{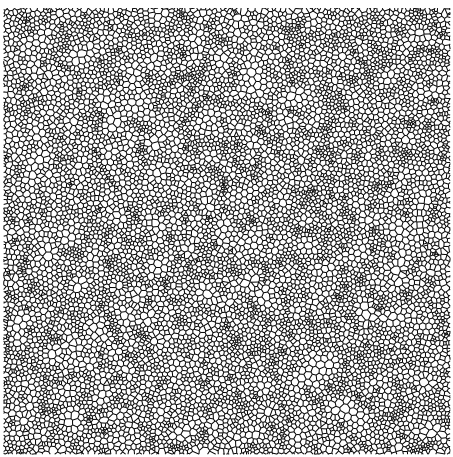}& \hspace{-40pt}
\includegraphics[width=6cm,height=6cm]{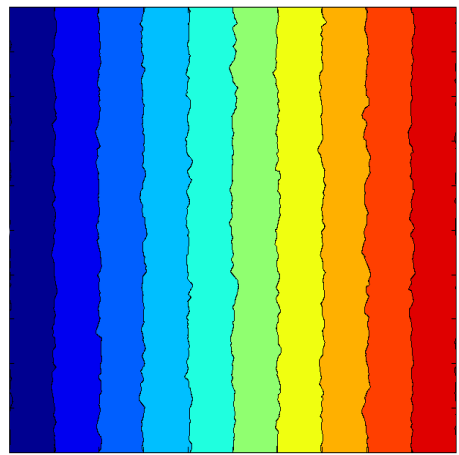}
\end{array}
\]
\caption{Some harmonic pictures	}
\label{dibus}
\end{figure}

\begin{figure}
\[
\begin{array}{c}
\includegraphics[height=9cm]{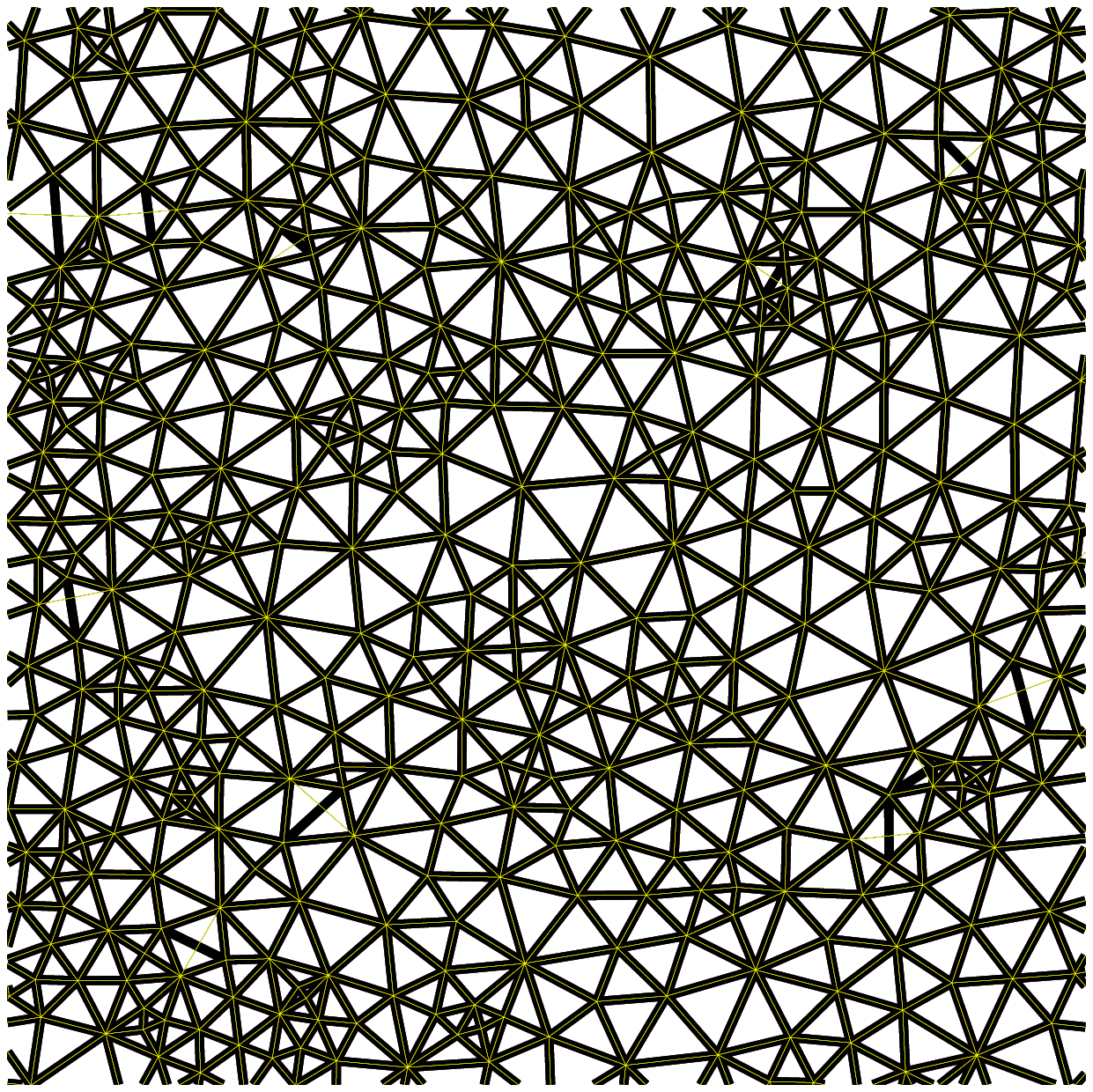}
\end{array}
\]
\caption{Delaunay triangulation of harmonic points (black and thick dashed) vs.
harmonic graph (yellow)}
\label{dibus2}
\end{figure}

\section{Appendix. The Random Walk and the Environment Process}

This appendix collects some technical results used in Section
\ref{harness}. The {\em environment seen from the particle} was used by De Masi
et al.\/ \cite{FerrariWick} to show the annealed invariance principle for the
random walk in the supercritical bond-percolation cluster. We adapt some of
those results to our setting.

Let $\ss\in \Nlf$ and $s\in
\ss$. Let $\dX^{\vv}_n$ be a discrete time random walk on $\ss$ with law
$\dP_\ss$ defined by $\dX^\vv_0=\vv$ and for $n\ge 1$,
\[
\dP_\ss(\dX^\vv_n=\vv''|\dX^\vv_{n-1}=\ww)=\frac{a(\ww,\vv'',\ss)}{a(\ww,\ss)}.
\]
That is, the walk starts at $s$ and if it is at $s'\in\ss$, then it
chooses a neighbor uniformly at random and jumps over it. Let $\dE_\ss$ be
the expectation with respect to $\dP_\ss$.

To build the continuous time walk, let $N=\{T_k; k \in\N\}$ be a rate $1$
homogeneous Poisson Process in $\R_+$, independent of $(\dX_n)_{n\geq 0}$, and
define
\begin{equation}\label{passeiocontinuo}
X_t:=\dX_{N(t)},
\end{equation}
where $N(t)=|N\cap (0,t]|$ is the number of points of $N$ in the interval
$(0,t]$. Let $P_\ss=\dP_\ss\otimes Q$, where $Q$ is the law of $N$ in
$(\Nlf(\R^+),\B(\Nlf(\R^+)))$.  The law of $X^0_t$ coincides with the law of the
walk $B^0_{[t,0]}$ defined in Section \ref{construction}, so that the results
below hold for $B^0_{[t,0]}$.

Given the process $\dX^{\vv}_n$ (with initial state $\vv\in\ss$), define the
process \[\ss_n^\circ=\tau_{\dX^s_n}\ss.\] This process can be thought as the
\emph{environment as seen from the particle} moving according to $\dX^{\vv}_n$. The
process $\ss_n^\circ$ is Markov with values in $\Nlf^\circ$ (i.e. for all
$n$, $0\in\ss_n^\circ)$. We
use $P_\ss$ to denote the law of $\ss_n^\circ$ in $\Nlf^{\Z^+}$ with initial
state $\ss$ .

Let $\M$ be the set of aperiodic $\ss$: 
\begin{equation}\label{eq:espaco_dif}
\M=\{\ss\in \Nlf\colon \tau_x\ss\neq \ss\mbox{ for all }x\in \R^d,\ x\neq 0\}.
\end{equation}
If $\ss$ is aperiodic, then the trajectory of $\ss_n$
determines univoquely the trajectory of the walk $\dX^0_n$. 
The Poisson Process is aperiodic almost surely.

Let $S$ be an ergodic point process in $\R^d$, with Palm version $\S$. Denote
by $\Q$ the probability measure on $(\Nlf,\B(\Nlf))$
given by
\[
\int f(\ss) \Q(d\ss)=\frac{1}{\E a(0)}\E [a(0)f(\S)].
\]
for bounded measurable $f\colon \Nlf\to \R$.

\begin{Lemma}\label{meio_particula}
The process $(\ss_n^\circ)_{n\geq0}$ is reversible and ergodic under $\Q$.
\end{Lemma}

\begin{proof}
  To check reversibility, let $f,g\colon \Nlf\to \R$ be bounded measurable
  functions. Define $\phi(\vv,\ww,\S)=a(\vv,\ww,\S)f(\tau_\vv\S)g(\tau_\ww\S)$
  and observe that $\phi$ is covariant and integrable, and therefore, by the
  Mass Transport Principle (Lemma \ref{masstransport})
  \begin{align*}
    \int E_\sso f(\sso)g(\ss_1^\circ)\Q(d\sso)
    &=(1/\E a(0))\E\sum_{\vv\in\S}a(0,\vv,\S)f(\S)g(\tau_\vv\S)\\
    &=(1/\E[ a(0)])\E\sum_{\vv\in\S}a(0,\vv,\S)f(\tau_\vv\S)g(\S)\\
    &=\int E_\sso f(\ss_1^\circ)g(\sso)\Q(d\sso).
  \end{align*}
  To show ergodicity, let $A\in\B(\Nlf^\circ)$ be an invariant set for the
  dynamics, that is $A$ is such that $\ss_0^\circ\in A$ implies $\ss_1^\circ\in
  A$. This implies that for any neighbor $s$ of the origin, $\tau_s\ss_0^\circ
  \in A$. Iterating the argument one shows that, if $\ss^\circ\in A$ then
  $\tau_\vv\sso\in A$ for every $\vv\in\sso$. Therefore, 
\[
\Pe_o(A)
:=\Pe(\S\in A)=\lim_{\Lambda\nearrow\R^d}\frac{1}{|\Lambda|}
\sum_{\vv\in\S}\one_{\tau_\vv\S\in A}\in\{0,1\}
\] 
and, as $\Q\ll\Pe_o$ and $\Pe_o\ll\Q$, it follows that
$\Q(A)\in\{0,1\}$.
\end{proof}

\begin{Proposition}\label{prop:momentos-passeio}
  Let $r\geq 1$ and $\gamma$ be a surface with covariant gradient. If
  $c:=2\C(|\nabla\gamma|^r)<\infty$  then
\[
\EE|\gamma(X_t)-\gamma(X_0)|^r
\;=\;\E E_{\S}|\gamma(X_t)-\gamma(X_0)|^r
\;\leq\; \E( a(0)E_{\S}|\gamma(X_t)-\gamma(X_0)|^r)\;\leq\; c m^r(t),
\]
where $m^r(t)$ is the $r$-th moment of a Poisson random variable with mean $t$.
\end{Proposition}

\begin{proof}
Suppose, without loss of generality, that $\gamma(0)\equiv0$, and observe that
\begin{align*}
\E(E_{\S}|\gamma(X_t)|^r)
&=\sum_{n=1}^\infty\E(E_{\S}|\gamma(\tilde{X}_n)|^r\one_{N(t)=n})
=\sum_{n=1}^\infty\E(E_{\S}|\gamma(\tilde{X}_n)|^r)e^{-t}\frac{t^n}{n!}
\end{align*}
On the other hand, by H\"older's inequality
\begin{align*}
|\gamma(\tilde{X}_n)|^r	
&=	|\sum_{k=1}^n(\gamma(\tilde{X}_k)-\gamma(\tilde{X}_{k-1}))|^r
\leq n^{r-1}\sum_{k=1}^n|\nabla\gamma(\tilde{X}_{k-1},\tilde{X}_{k})|^r.
\end{align*}
Finally, as $\tilde{X}_{k}-\tilde{X}_{k-1}$ depends only on $\ss_k$ and
$\ss_{k-1}$, by the stationarity of $\ss_n$ under $\Q$ (Lemma
\ref{meio_particula}), it follows that
\begin{align*}
\E(E_{\S}|\gamma(X_t)|^r)
&\leq \sum_{n=1}^\infty n^{r-1}
\sum_{k=1}^n\E(E_{\S}|\nabla\gamma(\tilde{X}_{k-1},\tilde{X}_{k},\S)|^r)e^{-t}\frac{t^n}{n!}\\
&\leq \sum_{n=1}^\infty n^{r-1}
\sum_{k=1}^n\E(a(0)E_{\S}|\nabla\gamma(0,\tilde{X}_{k}-\tilde{X}_{k-1},\ss_{k-1})|^r)e^{-t}\frac{t^n}{n!}\\
&= \sum_{n=1}^\infty n^{r}\E(a(0)E_{\S}|\nabla\gamma(0,\tilde{X}_{1},\S)|^r)e^{-t}\frac{t^n}{n!}\\
					&= \E(\sum_{\vv\in \S}a(0,s)|\nabla\gamma(0,\vv,\S)|^r)m^r(t). \qedhere
\end{align*}
\end{proof}

To obtain estimates for $\C(|\nabla\eta_t^\gamma|^r)$ we study the process of
the \emph{environment as seen from the random walker} on $\S$, as in
\cite{FerrariWick}.  The law of $\S$ is reversible and ergodic for this process,
which allows us to make estimates on the original random walk. Let $\rw$ as in Section \ref{construction} be a random walk on the points of $\S$
starting at $0\in\S$, and denote its law by $P_\S^0$. 

\begin{proof}[Proof of Lemma \ref{momentosprocesso}]
  From the covariance of $\nabla\gamma$ we can assume, without loss of generality, that $\gamma(0)\equiv 0$. By the Mass Transport Principle Lemma \ref{masstransport} and
  Proposition \ref{prop:momentos-passeio},
\begin{align*}
  \C(|\nabla\psi_t|^r)&=\;\frac{1}{2}\EE\sum_{\vv\in\S}a(0,\vv)|\nabla\psi_t(0,\vv)|^r\;
  \leq\; 2^{r-2}\EE\sum_{\vv\in\S}a(0,\vv)[|\psi_t(0)|^r+|\psi_t(0)|^r]\\
  &\leq\; 2^{r-1}\EE a(0)|\psi_t(0)|^r
  \;\leq\; 2^{r-1}\EE a(0)|\gamma(B^0_{[t,0]})|^r
  \;\leq\; 2^r\C(|\nabla\gamma|^r)m^r(t) \quad \Pe\mbox{-a.s.}. \qedhere
\end{align*} 
\end{proof}

The following Lemma is a part of the proof of Theorem 2.1 in \cite{FerrariWick};
we omit the proof.

\begin{Lemma}\label{lema:seg_mom_disc}
If a surface $\gamma$ satisfies
\[
\E\sum_{\vv\in S}a(0,s)|\gamma(\vv)|^2<\infty,
\]
then
\begin{equation}\label{eq:varquaddisc}
\lim_{n\to\infty}\frac{\E( a(0)E|\gamma(\dX_n)|^2)}{n}<\infty, \qquad\hbox{ and }\qquad
\lim_{t\to\infty}\frac{\E( a(0)E|\gamma(X_t)|^2)}{t}<\infty.
\end{equation}

\end{Lemma}

\bigskip
\noindent{\bf Acknowledgments }
We thank Marek Biskup for his comments and to
the referee for the carefully reading of the manuscript and for many comments that
helped us to improve the paper.

\bibliographystyle{amsplain}
\bibliography{bibliografia}

\end{document}